\documentclass[10pt]{amsart}

\pdfoutput=1

\usepackage[footnotesize,bf]{caption}
\usepackage[left=1.25in,right=1.25in,top=1in]{geometry}

\usepackage{amsmath,amssymb,bbm,amsthm}
\usepackage{graphicx}
\usepackage{verbatim}
\usepackage{mathrsfs}

\usepackage{shortcuts}

\usepackage{array,multirow,booktabs} 

\usepackage{chngcntr}
\usepackage[colorlinks=true]{hyperref}
\usepackage{cleveref}

\newcommand{\bs}[1]{\boldsymbol{#1}}

\newcommand{\bsk}{\boldsymbol{k}}

\newcommand{\bsz}{\boldsymbol{z}}
\newcommand{\bsj}{\boldsymbol{j}}
\newcommand{\bsn}{\boldsymbol{n}}

\newcommand{\bsm}{\boldsymbol{m}}
\newcommand{\bPsi}{\bs{\Psi}}

\newcommand{\bD}{\bs{D}}
\newcommand{\bsX}{\bs{X}}
\newcommand{\bW}{\bs{W}}

\newcommand{\bsc}{\bs{c}}

\newcounter{parentnumber}

\makeatletter
\newenvironment{subtheorem}[1]{%
  \counterwithin*{#1}{parentnumber}
  \def\subtheoremcounter{#1}%
  \refstepcounter{#1}%
  \protected@edef\theparentnumber{\csname the#1\endcsname}%
  \setcounter{parentnumber}{\value{#1}}%
  \setcounter{#1}{0}%
  \expandafter\def\csname the#1\endcsname{\theparentnumber.\Alph{#1}}%
  \ignorespaces
  \def\foobar{#1} 
}{%
  \setcounter{\subtheoremcounter}{\value{parentnumber}}%
  \counterwithout*{\foobar}{parentnumber} 
  \ignorespacesafterend
}
\makeatother

\crefname{theorem}{Theorem}{Theorems}

\title{Stochastic collocation methods via $L_1$ minimization using randomized quadratures}

\author{Ling Guo}
\thanks{Ling Guo. Department of Mathematics, Shanghai Normal University, Shanghai, China. L.~Guo was partially supported by NSFC (11101287),
Shanghai Leading Academic Discipline Project (No. S30405) and E¨CInstitutes of Shanghai Municipal Education Commission (E03004)}

\author{Akil Narayan}
\thanks{Akil Narayan. Mathematics Department and Scientific Computing and Imaging Institute, University of Utah, University of Utah, Salt Lake City, UT 84112. A.~Narayan was partially supported by AFOSR FA9550-15-1-0467 and DARPA N660011524053}

\author{Tao Zhou}
\thanks{Tao Zhou. LSEC, Institute of Computational Mathematics and Scientific/Engineering Computing, AMSS, Chinese Academy of Sciences, Beijing, China. T.~Zhou work was supported the National Natural Science Foundation of China (Award Nos. 91530118 and 11571351).}

\author{Yuhang Chen}
\thanks{Yuhang Chen. Mathematics Department and Scientific Computing and Imaging Institute, University of Utah, University of Utah, Salt Lake City, UT 84112.}

\begin{document}
\maketitle

\renewcommand{\thefootnote}{\fnsymbol{footnote}}

\begin{abstract}
In this work, we discuss the problem of approximating a multivariate function by polynomials via $\ell_1$ minimization method, using a random chosen sub-grid of the corresponding tensor grid of Gaussian points.  The independent variables of the function are assumed to be random variables, and thus, the framework provides a non-intrusive way to construct the generalized polynomial chaos expansions, stemming from the motivating application of uncertainty quantification. We provide theoretical analysis on the validity of the approach.  The framework includes both the bounded measures such as the uniform and the Chebyshev measure, and the unbounded measures which include the Gaussian measure.  Several numerical examples are given to confirm the theoretical results.
\end{abstract}



\pagestyle{myheadings}
\thispagestyle{plain}
\markboth{Stochastic $\ell_1$-minimization using randomized quadratures}{}

\section{Introduction}

Stochastic computation has received intensive attention in recent years, due to the pressing need to conduct uncertainty
quantification (UQ) in practical computing. One of the most widely used techniques in UQ is generalized polynomial chaos
(gPC), see e.g. \cite{Ghanem,UQ_book,XiuK2}. In gPC, the stochastic function $f(x)$, with $x\in R^d$, is approximated via $d$-
variate orthogonal polynomials, whose orthogonality is defined by the probability
measure of the input variable $\mathbf{x}$. This becomes one of the most effective procedures
for UQ, and many numerical techniques on how to construct gPC approximations
have been developed. For practical computing, the gPC
stochastic collocation algorithm is highly popular because it allows one to repetitively use existing deterministic
simulation codes and to render the construction of gPC approximation a post-processing step. In the stochastic
collocation framework, one seeks to construct a gPC type orthogonal polynomial approximation
via point-evaluations of $f$. Popular methods for achieving this include sparse grids approximation \cite{BNT,Agarwal2009,bieri,eldred,FWK,GZ,MaZabaras,FabioC3}, pseudo orthogonal projection \cite{NajmGhanem} ,
interpolation \cite{narayan_stochastic_2012}, and least squares approach \cite{Tang_zhou_2014,Zhou_Narayan_Xu,CCMNT,NJZ_MC} to name a
few. For a review of the gPC methodology, see \cite{narayan_stochastic_2015}.

The challenge is in high-dimensional spaces, where the number of collocation nodes grows fast. Since each node
represents a full-scale deterministic simulation, the total number of nodes one can afford is often limited, especially
for large-scale problems. This represents a significant difficulty in constructing a gPC-type approximation using the
existing approaches: it is often not possible to construct a good polynomial approximation using a very limited number
of simulations in a large dimensional random space.

A more recent development in signal analysis is compressive sensing, also known as compressed sampling. Compressive
sensing (CS) deals with the situation when there is insufficient information about the target function. This
occurs when the number of samples is less than the cardinality of the polynomial space for the approximation. CS then
seeks to construct a polynomial approximation by minimizing the norm of the polynomial, typically its $\ell^1$ norm or $\ell^0$ norm (count the nonzero components). Following the seminal work of \cite{candes1,candes2,CandesTao_2005,Donoho_2006}, the theory of CS has generated an enormous
amount of interest in many disciplines and resulted in many newer theoretical results and practical implementations. The success of the CS methods lies in the assumption that in practice many target functions (signals) are sparse,
in the sense that what appear to be rough signals in the time/space domain may contain only a small number of notable. The works of \cite{schwab1,schwab2,doostan1,CDS_2011} show that under weak assumptions, solutions to the stochastic elliptic PDEs exhibit sparsity
in the properly chosen basis.

In a recent work \cite{doostan1}, the idea of CS has been extended to stochastic collocation and resulted in a highly flexible
method. With CS, one can employ arbitrary nodal sets with an arbitrary number of nodes. This can be very helpful in
practical computations. In \cite{doostan1}, some key properties, such as the probability under which the sparse random response
function can be recovered, are studied. Then Yan, Guo and Xiu extends this work in \cite{XiuL1}, which focus on the recoverability of stochastic
solutions in high-dimensional random spaces with random sampling. This is relevant because in UQ simulations the dimensionality is often
determined by the number of random parameters and can be very large. The work in \cite{doostan2} concerned with
convergence analysis and sampling strategies to recover a sparse stochastic function in both Hermite and Legendre
PC expansions from $\ell_1$-minimization problem. Although random sampling methods have been widely used in the CS
framework, a judicious, deterministic choice of points may provide several advantages over randomly-generated points. In \cite{XUZHOU},
the authors use Weil points to recover sparse Chebyshev polynomials. Tang and Zhou proposed a structured sampling method on quadrature points
to recover sparse Legendre polynomials in \cite{TangIa}.

In this work, we investigate the compressive sensing approach for stochastic collocation method, using a random chosen sub-grid from the corresponding tensor grid of Gaussian points. We will provide theoretical analysis on the convergence of such an approach. The main contribution of this work is to show that the method proposed in \cite{TangIa} is actually applicable for most random variables of interest, including measures on bounded domains (e.g., the uniform and Beta distributions) and unbounded domains (e.g., normal distributions). We also provide with several numerical examples to confirm the theoretical results.

Our main result is a sample count criterion for sparse recovery for general polynomial spaces with general measures. The precise statement is given by Theorem \ref{thm:main-sample-count} and can be summarized as follows: suppose we wish to recover a multivariate polynomial expansion of $f(\bs{X})$ from a finite-dimensional subspace of dimension $N$ whose maximum polynomial degree in any dimension is less than $n$. To do this, we use $M$ collocation samples and perform an $\ell^1$ minimization procedure. The $M$ samples are chosen randomly (with the uniform probability law) from a tensor-product Gaussian quadrature grid with $n$ points in each dimension. This method produces an approximation comparable to the best $s$-term expansion of $f$ with high probability if
\begin{align*}
  M \gtrsim L(n) s
\end{align*}
where
\begin{itemize}
  \item $L(n) \leq C^d$ if $\bs{X}$ has independent components, with each component having a Beta distribution.
  \item $L(n) \leq (C n)^{2d/3}$ if $\bs{X}$ is normally-distributed.
  \item $L(n) \leq (C n)^{2d/3}$ if $\bs{X}$ has independent components, with each components having a one-sided exponential distribution.
\end{itemize}

In all cases the constant $C$ is independent of the degree $n$ and the dimension $d$, and depends only on the (one-dimensional) marginal distributions of $\bs{X}$. In the first case, the $C^d$ dependence was established in \cite{TangIa} assuming the special case of $\bs{X}$ a uniform random variable on a hypercube. In the latter two cases where $\bs{X}$ has unbounded state space, the $n^{2d/3}$ dependence seems unpleasant, but this rate is essentially sharp if one uses our strategy for analysis and insists on sub-sampling from a tensor-product Gauss quadrature grid. Despite this dependence, we note that our analysis is quite general, extending to any random vector $\bs{X}$ whose independent components have Beta or Exponential (one- or two-sided) distributions with essentially any shape parameters. Finally, we would like to mention the work by Li and Zhang \cite{Li_Zhang}, where a interpolation scheme (with $M=N$) on a sub-set of Gaussian quadrature points is proposed, and the sub-set therein is chosen according to the corresponding value of the joint probability density function of the random input. We expect that the idea in \cite{Li_Zhang} may be useful also for the compressed sensing approach, and this will be investigated in our future studies.

The rest of the paper is organized as follows. In Section 2, we introduce the gPC approximation, set up the $\ell_1$-minimization problem and definitions and theorems used in the compressed sensing approach. Section 3 gives a short view of Gaussian quadrature, introduces the discrete transform as a discretization of orthogonal polynomial system, and gives the theorems for recovering polynomial chaos when sampling from the discretized system. Several numerical tests are provided in Section 4, and we finally give some conclusions in Section 5.

\section{The setup}
Let $\bs{X}=(X^1 , \ldots , X^d)^\top$ be a random vector with $d$ mutually independent components; each $X^i$ takes values in $\Gamma^i \subset \mathbb{R}.$ Since the variables $\{X^i\}_{i=1}^d$ are mutually independent, they have marginal probability density functions $\rho^i$ associated with random variable $X^i$ that completely characterize the distribution of $\bs{X}$. Define $\Gamma:= \otimes_{i=1}^d\Gamma^i \subset \mathbb{R}^d,$ and let $\rho(\mathbf{x})= \prod_{i=1}^d \rho^i(x^i): \Gamma\rightarrow \mathbb{R}^+$ denote the joint probability density function (PDF) of $\bs{X}.$

In a simple stochastic collocation setup, we consider a $d$-variate function $f:\Gamma \to \mathbb{R}$, and wish to recover information about this function from a finite (ideally small) set of function evaluations. Let $\theta_M=\{\mathbf{x}_1,\cdot\cdot\cdot,\mathbf{x}_M\}\in\Gamma$ be a set of points at which the function values of $f$ are available, and denote these values as  $f_m=f(\mathbf{x}_m), m=1,\cdot\cdot\cdot,M$. We are concerned with the approximation of the function $f$ based on $\{\mathbf{x}_m,f_m\}_{m=1}^{M}$.  Let $V$ be a linear space from which the approximation is sought and let $N=\text{dim} V$. This paper is concerned with the special case of the above setup where $V$ is a polynomial subspace, and the $N = \dim V$ coefficients defining the approximation are underdetermined from $M < N$ samples.

Much of the rest of this section is tasked with introducing our requisite notation. A summary of much of this notation is given in Table \ref{tab:notation}.

\begin{table}
  \begin{center}
  \resizebox{\textwidth}{!}{
    \renewcommand{\tabcolsep}{0.4cm}
    \renewcommand{\arraystretch}{1.3}
    {\scriptsize
    \begin{tabular}{@{}cp{0.8\textwidth}@{}}\toprule
      Symbol(s)    & \\ \midrule
      $d, i$       & Dimension $d$, index $1 \leq i \leq d$. \\
      $\bsn, n_i$  & Multivariate tensor-product quadrature rule size $\bsn$, with $n_i$ points in dimension $i$ \\
      $\bsX, X^i$  & Random vector $\bs{X}$, components $X^i$ \\
      $\Gamma, \Gamma^i$  & State space of $\bs{X}$ and $X^i$, respectively \\
      $\rho, \rho^i$  & Joint density of $\bs{X}$, marginal density of $X^i$, respectively \\
      $\phi_{\bsk}, \varphi^i_{k_i}$ & Multivariate gPC basis element of degree $|\bsk|$, dimension-$i$ univariate gPC basis element of index $k_i$, respectively. Each element has polynomial degree $|\bsk|$, and $k_i$, respectively.\\
      $\Lambda, \Lambda^P_{\bsn}$ & General multi-index set, multi-index set $\left\{ \bsk\in \N_0^d \, | \, \bsk \leq \bsn \right\}$, respectively.\\
    \bottomrule
    \end{tabular}
  }
    \renewcommand{\arraystretch}{1}
    \renewcommand{\tabcolsep}{12pt}
  }
  \end{center}
  \caption{Notation used throughout this article.}\label{tab:notation}
\end{table}

\subsection{Generalized polynomial chaos}
We are primarily concerned with the approximation $f$ in a polynomial subspace $V$. In particular, we seek this approximation using the Generalized polynomial chaos (gPC) framework. The basic idea of gPC is to represent the function $f$ as a polynomial of the random variables $X^i$; the basis for approximation is typically taken as a set of orthogonal polynomials. For each marginal density $\rho^i,$  we can define the univariate gPC basis elements $\varphi^i_n$, polynomials of degree $n$, via the orthogonality relation
\begin{align}\label{eq:phi-orthonormality}
  \mathbb{E} \left[\varphi^i_n(X^i) \varphi^i_{\ell}(X^i)\right] = \int_{\Gamma^i} \varphi^i_n(s) \varphi^i_{\ell}(s) \rho^i(s) ds = \delta_{n,\ell}, \quad n,\,\,\ell \geq 0,
\end{align}
with $\delta_{n,\ell}$ the Kronecker delta function. Up to a multiplicative sign, this defines the polynomials $\varphi^i_n$ uniquely; thus the probability measure $\rho^i$ determines the type of orthogonal polynomial basis. For example, the Gaussian (normal) distribution yields the Hermite polynomials, the uniform distribution pairs with Legendre polynomials, etc. For a detailed account of the correspondence, see \cite{XiuK2}. In this paper, each marginal density $\rho^i$ may be associated with any of the classical orthogonal polynomial families: this includes Beta distributions (Legendre, Chebyshev, and Jacobi polynomials), normal distributions (Hermite polynomials), and exponential distributions (Laguerre polynomials).

For the multivariate ($d>1$) case, we will use standard multi-index notation. For some $\bsn \in \N_0^d$, it has components denoted $n_i$, i.e., $\bsn = \left(n_1, \ldots, n_d\right)$. Its modulus is its $\ell^1$ norm, i.e., the sum of its components $|\bsn| = \sum_{i=1}^d n_i$. The factorial is the product of its componentwise factorials, $\bsn ! = \prod_{i=1}^d n_i !$, where we take $0 ! = 1$. Given another $d$-dimensional index $\bsj$, we have
\begin{align*}
  \bsj + \bsn &= \left( j_1 + n_1, \ldots, j_d + n_d \right), & \bsn^{\bsj} &= \prod_{i=1}^d n_i^{j_i}.
\end{align*}
Boldface explicit constants, e.g., $\bs{3}$ are multi-indices with the explicit value repeated:
\begin{align*}
  \bs{3} = \left(3, \ldots, 3\right) \in \N^d_0.
\end{align*}
A partial ordering is defined on multi-indices:
\begin{align*}
  \bsj < \bsn \hskip 10pt \Longleftrightarrow \hskip 10pt j_i < n_i\;\; \forall\;\; i=1,\ldots, d,
\end{align*}
with a similar definition for $\leq$. We will occasionally use multi-index operations on $d$-dimensional vectors whose entries are non-integers; given the definitions above, the meaning of such notation should be clear.

One convenient representation for a multivariate gPC basis is as a product of the univariate gPC polynomials in each direction. For a multi-index $\bsn \in \N_0^d$, we have
\begin{align}
\label{gpcbais}
  \phi_{\bsn}(\mathbf{x}) = \prod_{i=1}^d \varphi^{i}_{n_i}\left(x^{i}\right).
\end{align}
The product functions $\phi_{\bsn}$ are $L^2$ orthogonal under the joint probability density function $\rho$ for $\mathbf{X}$:
\begin{align}
\label{othorgonal}
\mathbb{E} \left[\phi_{\bsn}(\mathbf{X}) \phi_{\bsj}(\mathbf{X})\right] = \int_{\Gamma} \phi_{\bsn}(\mathbf{x}) \phi_{\bsj}(\mathbf{x}) \rho(\mathbf{x}) d\mathbf{x} = \delta_{\bsn,\bsj},
\end{align}
where $\delta_{\bsn, \bsj} = \prod_{i=1}^d \delta_{n_i, j_i}$.

A gPC expansion for $f$ is given by
\begin{align}
  f(\mathbf{X}) = \sum_{\bsj\in \mathbb{N}_0^d}{c}_{\bsj} \phi_{\bsj}(\mathbf{X}),
\end{align}
where the coeficients $c_{\bsj}$ are the unknowns that must be computed from available knowledge of the function $f$. If $f(\bsX)$ has finite variance, then this expansion is well-defined and convergent in the $L^2$ sense under fairly weak assumptions on the distribution of $\V{X}$ \cite{ernst_convergence_2012}.

For computational purposes, the above gPC expansion must be truncated. One widely-used approach is to approximate $f$ in a \textit{total-degree} space $V = T_n$:
\begin{align*}
  \Lambda^T_{n} &= \left\{ \bsk \in \mathbb{N}_0^d\, | \, \sum_{i=1}^d k_i \leq n \right\}, & T_n &= \mathrm{span} \left\{ \phi_{\bsk}\, | \, \bsk \in \Lambda^T_{n} \right\}.
\end{align*}
The dimension of $T_n$ is
\begin{align}\label{eq:td-dim}
N=  \# \Lambda^T_{n} \triangleq \dim T_n = \left( \begin{array}{c} d+n \\ n \end{array}\right).
\end{align}
Another common index space is the \textit{tensor-product} index space; for a maximum degree $n \in \N_0$,
\begin{align*}
  \Lambda^P_{n} &= \left\{ \bsk \in \mathbb{N}_0^d\, | \, \max_i k_i \leq n \right\}, & P_n &= \mathrm{span} \left\{ \phi_{\bsk}\, | \, \bsk \in \Lambda^P_{n} \right\}.
\end{align*}
Note that the tensor-product space is larger than the total degree space, $\Lambda^T_n \subset \Lambda^P_n$, and that its cardinality is $\# \Lambda^P_n = (n+1)^d$.
We can define anisotropic versions of the tensor-product space: Given a multi-index $\bsn \in \N_0^d$, we can allow polynomials up to degree $n_i$ in dimension $i$,
\begin{align*}
  \Lambda^P_{\bsn} &= \left\{ \bsk \in \mathbb{N}_0^d\, | \, \bsk \leq \bsn \right\}, & P_{\bsn} &= \mathrm{span} \left\{ \phi_{\bsk}\, | \, \bsk \in \Lambda^P_{\bsn} \right\}.
\end{align*}
The total degree space $T_n$ is a more practical space to use in computations compared to the tensor-product space $P_n$ because the latter has extremely large dimension when $d$ or $n$ grows.

We remark that for any finite multi-index set $\Lambda$, it is always possible to define an ordering scheme such that the multi-indices can be ordered via a single index. That is, we have
\begin{align*}
  \{\phi_{\bsk}(\mathbf{x})\}_{\bsk \in \Lambda}\Leftrightarrow \{\phi_j(\mathbf{x})\}_{j=1}^{N},
\end{align*}
for some one-to-one correspondence between $j \in \{1, \ldots, N\}$, with $N = \# \Lambda$. For any such $\Lambda$, $V$ is the $N$-dimensional polynomial subspace formed from the span of $\phi_{\bsk}$ for $\bsk \in \Lambda$.

One of our main tasks is to estimate the projected function of $f$ in the total degree space $T_n$, i.e.
\begin{align}
\label{eq:pce}
f_n=\sum_{\bsk \in \Lambda^T_{n}}c_{\bsk}\phi_{\bsk}=\sum_{j=1}^{N}{c}_j \phi_j.
\end{align}
using interpolation conditions on some data $f(\mathbf{x}_m),m=1,...,M$.

\subsection{Compressive sensing approach} We now present the basic formulation for the stochastic collocation methods in the
compressed sensing framework. In the framework of this paper, the compressive sensing approach can be described as follows. Given a set of $M$ realizations $\{\mathbf{x}_{i}\}_{i=1}^M$
, with corresponding outputs $\mathbf{f}=[f({\mathbf x}_1),\ldots,f({\mathbf x}_M)]^\top $, we now seek a solution that satisfies
\begin{align}\label{eq:normal_equation}
  \bPsi \bsc = \bs{f}
\end{align}
where $\mathbf{c}= ({c}_1,\cdot\cdot\cdot,{c}_{N})^\top$ is the coefficient vector of gPC (\ref{eq:pce}),
and
\begin{align*}
\mathbf{\Psi}=(\Psi_{ij})_{1\leq i\leq M,1\leq j\leq N} \in \mathbb{R}^{M \times N} \quad  \Psi_{ij}=\phi_j(\mathbf{x}_i),
\end{align*}
is the Vandermonde-like matrix, often referred to as the design matrix.

This problem is determined when $M = N$, overdetermined when $M > N$, and underdetermined when $M < N$.
It is the underdetermined case that is considered here. This is often encountered in practice, especially in high
dimensions with the total degree space $T_K$ since $\dim T_K \sim K^d$ can be extremely large. In general when $d\gg 1$, the cardinality $N$ of most standard polynomial spaces becomes extremely large, even when the order of the polynomials is moderate. On the other hand, in many practical applications the evaluation of the target function $f({\mathbf x}_i)$ is expensive and one often has much smaller number of samples $M$ than the number of gPC coefficients $N$, i.e. $M\ll N$. Thus  problem (\ref{eq:normal_equation}) becomes ill-posed and we need some form of regularization to obtain a unique solution.

One efficient method uses an $\ell_1$-minimization algorithm which, under certain conditions, provides a means of identifying sparse coefficient vectors from a limited amount of data. A polynomial chaos expansion is defined as $s$-sparse when $\|\bs{c}\|_0\leq s$, i.e the number of non-zero coefficients, does not exceed $s$. An $\ell_1$-minimization scheme attempts to find dominant gPC coefficients by solving
the optimization problem
\begin{align}
\label{eq:l1minimization}
\text{argmin} \|\mathbf{c}\|_1 \quad \text{subject to} \quad \bPsi \bsc = \bs{f},
\end{align}
where $\|\cdot\|_{1}$ is the $\ell_1$ norm on vectors. The advantage of the above formulation (in contrast to optimizing over $\|c\|_0$) is that it is a convex problem,
and so computational solvers for convex problems may be leveraged (see \cite{Osher1,yangzhang}). Under certain conditions, the solution to the $\ell_1$ optimization problem also solves the $\ell_0$ problem (see Theorem \ref{thm:l1-l0-connection} below). This $\ell_1$ minimization problem is
often referred to as \textit{Basis Pursuit}. Other types of minimization problems such as \textit{Basis Pursuit Denoising} and \textit{Least
Absolute Shrinkage Operator} (LASSO) can be found in \cite{donoho2,tibs} and references therein.

\subsection{Recovery via $\ell_1$ minimization}
The ability of $\ell_1$-minimization method (\ref{eq:l1minimization}) to determine the dominant coefficients of the gPC expansion is determined by the properties of the measurement matrix $\bPsi$ and the sparsity of $\bsc$. We need some definitions to make this precise.

\begin{definition}
The error of the best $s$-term approximation of a vector $\mathbf c \in
R^N$ in the $\ell^p$-norm is defined as
\begin{equation}
\label{best_s_term}
\sigma_{s,p}(\mathbf c) = \inf_{\|\mathbf{y}\|_0 \leq s}\| \mathbf {y-c}\|_p.
\end{equation}
\end{definition}
Clearly, $\sigma_{s,p}(\mathbf c)=0$ if $\mathbf c$ is $s$-sparse.

\begin{definition}[restricted isometry constant\cite{candes1,candes2}]\label{def:RIP}
Let $\mathbf{D}$ be an $M\times N$ matrix. Define  the restricted
isometry constant (RIC) $\delta_s<1$ to be the smallest positive
number such that the inequality
\begin{eqnarray}
\label{RIPmatrix}
(1-\delta_s)\|\mathbf{c}\|_2^2 \leq \| \mathbf{Dc}\|_2^2 \leq
(1+\delta_s) \| \mathbf{c}\|_2^2
\end{eqnarray}
holds for all $\mathbf{c} \in R ^N$ of sparsity at most $s$. Then,
the matrix $\mathbf{D}$ is said to satisfy the $s$-restricted
isometry property (RIP) with restricted isometry constant $\delta_s$.
\end{definition}

\begin{theorem}[Sparse recovery for RIP-matrices \cite{candes2,RauhutWard,Cai+Wang+Xu_2010}]\label{thm:l1-l0-connection}
Let $\mathbf{D} \in R^{M\times N}$ be a matrix with RIC satisfying $\delta_s<0.307$.
For any given $\tilde{\mathbf{c}}\in R^N$, let $\mathbf{c}^{\#}$ be the solution of the $\ell_1$-minimization
\begin{equation}
\label{BP}
\text{argmin} \| \mathbf{c}\|_1 \qquad \mbox{subject to}\quad \mathbf{Dc}=\mathbf{D}\tilde{\mathbf{c}}.
\end{equation}
Then the reconstruction error satisfies
\begin{eqnarray}
\label{BPerror}
\| \mathbf{c}^{\#}-\tilde{\mathbf{c}}\|_2 \leq C
\frac{\sigma_{s,1}(\tilde{\mathbf{c}})}{\sqrt{s}}
\end{eqnarray}
for some constant $C>0$ that depends only on $\delta_{s}$. In
particular, if $\tilde{\mathbf{c}}$ is $s$-sparse then reconstruction is
exact, i.e., $\mathbf{c}^{\#}=\tilde{\mathbf{c}}$.
\end{theorem}

We consider a system $\left\{\psi_k(\bs{x})\right\}_k$ that is orthonormal with respect to a density $\nu(\bs{x})$ a \textit{bounded orthonormal system} if it satisfies:
\begin{align}\label{bos}
\max_{1 \leq k \leq N}\| \psi_k\|^2_{\infty}
=\max_{1 \leq k \leq N}\sup_{\mathbf x \in \supp \nu}|\psi_k(\mathbf x)|^2 \leq L(N) < \infty.
\end{align}
If the bound $L$ is independent of $N$, then we call it a {\em uniformly bounded orthonormal system}, satisfying
\begin{equation}\label{ubos}
\sup_{1 \leq k}\| \psi_k\|^2_{\infty}
=\sup_{1 \leq k}\sup_{\mathbf{x} \in \supp \nu}|\psi_k(\mathbf x)|^2 \leq L,
\end{equation}
for some $L \geq 1$. For such systems, there is a precise undersampling rate for which the coefficient vector
can be recovered with high probability when we solve (\ref{eq:l1minimization}).

\begin{theorem} [RIP for bounded orthonormal systems \cite{Rauhut,RauhutWard}]
\label{RIP_bos}
Let $\mathbf{D}\in R^{M\times N}$ be the interpolation matrix with
entries $\{d_{ij}=\psi_j({\mathbf x}_i)\}$ from (\ref{eq:l1minimization}),
where the points ${\mathbf x}_i, i=1,\dots, M$, are i.i.d. random samples
drawn from the orthogonalizing measure $\nu$ for the bounded
orthonormal system $\{\psi_j\}$ satisfying \eqref{bos}. For some $\delta>0,$ assuming that
\begin{equation}
\label{num_RIP}
M \geq C \delta^{-2} L s \log^3(s)\log(N),
\end{equation}
then with probability at least $1-N^{-\gamma\log^3(s)}$, the RIC
$\delta_s$ of $\frac{1}{\sqrt{M}} \mathbf{D}$ satisfies $\delta_s \leq
\delta$. Here the $C$, $\gamma>0$ are generic constants.
\end{theorem}

The Monte Carlo (MC) sampling method utilizing iid samples is very promising in the CS framework, and the result above is one of the foundational tools in this regard. The authors in \cite{RauhutWard} use the concept of bounded orthonormal systems to quantify recovery of sparse expansion coefficients in a univariate Legendre polynomial basis. They exploit the fact that weighting Legendre polynomials by the factor $(1 - x^2)^{1/4}$ makes these polynomials a uniformly bounded system. The strategy is then to perform recovery with weighted/preconditioned Legendre polynomials, sampling from the appropriate biased measure that retains orthogonality; in this case this biased measure is the Chebyshev measure.

\begin{subequations}\label{eq:notation}
This idea of preconditioning systems is a theme in recent $\ell_1$ recovery procedures. Let $\bPsi$ denote an $M \times N$ matrix containing evaluations of the $N$ multivariate polynomials $\phi_n$ at collocation locations $\mathbf{x}_m$, with $\bW$ a diagonal $M \times M$ matrix:
\begin{align}\label{eq:psi-w-def}
  (\Psi)_{m,n} &= \phi_n(\mathbf{x}_m), & (W)_{m,m} &= w_m > 0,
\end{align}
where $w_m$ are weights that will be specified. We can introduce an $M \times N$ weighted matrix $\bD$ that serves as the matrix used to recover sparse coefficients in the weighted setting, defined as
\begin{align}\label{eq:D-def}
  \bD = \sqrt{\bW} \bPsi.
\end{align}
\end{subequations}
For an unknown function $f(\mathbf{x})$, the $M \times 1$ vector $\bs{f}$ with entries $(f)_m = f(\mathbf{x}_m)$ contains evaluations of $f$. The standard ``unweighted" $\ell_1$ optimization solves,
\begin{align}
\label{eq:unweighted-l1minimization}
\text{argmin} \|\mathbf{c}\|_1 \quad \text{subject to} \quad \bPsi \bsc = \bs{f},
\end{align}
whereas a weighted version is given by
\begin{align}
\label{eq:weighted-l1minimization}
\text{argmin} \|\mathbf{c}\|_1 \quad \text{subject to} \quad \mathbf{D} \mathbf{c} = \sqrt{\bW} \mathbf{f}.
\end{align}
Note that we introduce $\bW$ as a matrix of positive weights but use only its square root in the formulations \eqref{eq:D-def} and \eqref{eq:weighted-l1minimization}. While cumbersome at present, this choice will be notationally convenient later.

We can now summarize some existing methods for sparse recovery of multivariate Legendre expansions. The main results for recovery in the formulations \eqref{eq:unweighted-l1minimization} and \eqref{eq:weighted-l1minimization} attempt to show that the system matrices ($\bPsi$ and $\bD$, respectively) satisfy the conditions of Theorem \ref{RIP_bos} and thus can invoke Theorem \ref{thm:l1-l0-connection} to show convergence of the algorithm.

The authors in \cite{XiuL1} extend the univariate results of \cite{RauhutWard} to high-dimensional problems, for both the original (i.e., unweighted) $\ell_1$-minimization and the preconditioned (i.e., weighted) $\ell_1$-minimization.
\begin{theorem} [Recoverability with multivariate Legendre Polynomials \cite{XiuL1}]
\label{recovery_Legendre}
Let $\{\phi_j\}_{j=0}^{N-1}$ be the multivariate Legendre polynomial basis elements of the total degree space $T_n^d$, and let
$f(\mathbf{x})=\sum\limits_{j=0}^{N-1}\widetilde{c} \phi_j$ be an arbitrary polynomial with coefficient vector
$\widetilde{\bsc}$. For some nodal array $\{\mathbf{x}_i\}_{1\leq i\leq M}$, let the weights $w_m$ in \eqref{eq:notation} be defined by 
\begin{align}\label{eq:chebyshev-weights}
  w_{m}=\bigg (\frac{2}{\pi}\bigg )^{d}\prod\limits_{n=1}^d\bigg (1-(\mathbf{x}_i^n)^2\bigg )^{1/2}.
\end{align}
1. Assume $d \geq n$. With $\{\mathbf{x}_i\}_{1\leq i\leq M}$ i.i.d random samples drawn from the uniform measure on $[-1,1]^d$, and if
\[
M> 3^ns\log^3(s)log(N),
\]
then with high probability the solution $\bs{c}^\sharp$ to the direct $\ell_1$ minimization problem \eqref{eq:unweighted-l1minimization} is within
a factor of the best s-term error:
\[
Pr\bigg [\|{\mathbf{c}}^\sharp-\widetilde{\bsc}\|_2 \leq \frac{C\sigma_s(\widetilde{\bsc})_1}{\sqrt{s}}\bigg ]\geq 1-N^{-\gamma\log^3(s)}
\]
2. Let $\{\mathbf{x}_i\}_{1\leq i\leq M}$ be i.i.d random samples drawn from the Chebyshev measure, and assume that
\[
M> 2^ds\log^3(s)log(N).
\]
Then with high probability, the solution $\bs{c}^\sharp$ to the preconditioned/weighted $\ell_1$ minimization problem \eqref{eq:weighted-l1minimization}
 is within a factor of the best s-term error:
\[
Pr\bigg [\|\bsc^\sharp-\widetilde{\bsc}\|_2 \leq \frac{C\sigma_s(\widetilde{\bsc})_1}{\sqrt{s}}\bigg ]\geq 1-N^{-\gamma\log^3(s)}
\]
For both of the above cases, the constants $C$ and $\gamma$ are universal.
\end{theorem}

\section{Random sampling of Gaussian quadrature points}\label{sec:gpc-stuff}
We now present the method of random Gauss quadrature for sparse polynomial recovery via $\ell_1$ minimization problem. The basic idea is to use the results from the previous section to conclude that subsampling a tensor-product Gaussian quadrature produces an accurate recovery procedure for CS. The two basic ingredients are (i) Tensor-product Gaussian quadrature can be used to define discrete measures under which polynomials are orthogonal, (ii) weighted polynomials under a Gaussian quadrature rule have quantifiable bounds.

\subsection{Tensor grid of Gaussian points}
Let $\varphi_{n}^i(x^i)$ be the degree-$n$ orthonormal polynomial corresponding to the density $\rho^i$. It is well known that $\varphi_{n}^i$ has $n$ real and distinct zeros. I.e., there are $n$ distinct nodes $z^i_k$, $k=1, \ldots, n$, such that
\begin{align}\label{eq:gq-nodes}
 \varphi_{n}^i(z_k^i)=0, \quad k=1,\cdot\cdot\cdot,n.
\end{align}
Furthermore, an interpolatory quadrature with weights $w_k$ rule can be constructed on the zeros, satisfying
\begin{align}\label{eq:gq-exactness}
  \sum\limits_{k=1}^{n}w_k^if(z_k^i) &= \int_{X^i}f(x^i)\rho^i(x^i)dx^i,
\end{align}
for any polynomial $f$ of degree $2 n - 1$ or less; this is the Gaussian quadrature rule.\footnote{Note that  $z^i_k$ depends on the value of $n$, but we omit explicit notation indicating this dependence.} Here $w_k^i, k=1,\ldots,n$, are the univariate Gauss quadrature weights associated with dimension $i$. The $n$-point Gauss quadrature weights can be computed explicitly as
\begin{align}\label{eq:quadratureweight}
 w_k^i = \lambda_n^i(z_k^i), \quad k=1,\ldots,n,
\end{align}
where $\lambda_n^i$ is the 2-norm Christoffel function associated with dimension $i$:
 \begin{align}\label{eq:quadraturelamda}
 \lambda_n^i(x^i)=\frac{1}{\sum\limits_{k=0}^{n-1}(\varphi_k^i(x^i))^2}.
\end{align}
We define a discrete probability measure defined on the support of the $n$-point Gauss quadrature nodes. With $\delta_z$ the Dirac measure centered at $z$,
\begin{align}\label{eq:marginal-gq-measure}
  \nu^i_n \triangleq \frac{1}{n} \sum_{k=1}^{n} \delta_{z^i_k}
\end{align}
The measure $\nu^i_n$ is the uniform empirical probability measure associated with the discrete set $\left\{z^i_1, \ldots, z^i_{n}\right\}$.

We proceed to tensorize the univariate Gaussian quadrature rules. Let
 \begin{align}\label{eq:points}
 \Theta_{n}^i = \{z_1^i,. . .,z_{n}^i\}\subset \Gamma^i, \quad i=1,...,d
\end{align}
be the one dimensional $n$-point Gauss set associated to the $i$'th dimension. We then take tensor products to construct a $d$-dimensional point set. Let $n_i$ for $i=1, \ldots, d$ denote the Gauss quadrature rule size for dimension $i$. We collect these sizes into the multi-index $\bs{n} = \left( n_1, \ldots, n_d\right) \in \N^d$. The tensor product set is then
\begin{align}\label{eq:tensorpoints}
  \Theta_{\bs{n}}=\Theta_{n_1}^1\otimes\cdot\cdot\cdot \otimes\Theta_{n_d}^d
\end{align}
The cardinality of this set is $\left|\Theta_{\bsn}\right| = \prod_{i=1}^d n_i$. As before, an ordering scheme can be employed to order the points via a single index, i.e. for each $j=1,...,\left|\Theta_{\bsn}\right|$,
\begin{align*}
  z_j \leftrightarrow \bsz_{\bs{k}} &= (z_{k_1}^1,...,z_{k_d}^d), & \bs{1} \leq \bs{k} &\leq \bs{n}
\end{align*}
Each point has the scalar weight
\begin{align}\label{eq:tp-weights}
  w_{\bs{k}} &= \prod_{i=1}^d w_{k_i}^i, & \bs{1} \leq \bs{k} &\leq \bs{n}
\end{align}
This tensorized Gauss quadrature in the $d$-dimensional space $\Gamma$ exactly integrates any polynomial in the tensor space $P_{2\bs{n}-1}$. We note that using (\ref{eq:quadratureweight}) and (\ref{eq:quadraturelamda}), the $d$-dimensional Gauss quadrature weights for $\Theta_{\bs{n}}$ are given by
\begin{align}\label{eq:quadratureweight2}
  w_{\bs{k}} = \lambda_{\bs{n}}\left(z_{\bs{k}}\right) \triangleq \prod\limits_{i=1}^d\lambda_{n_i}^i(z_{k_i}^i) = \prod\limits_{i=1}^d\frac{1}{\sum\limits_{k=0}^{n_i-1}[\varphi_k^i(z_{k_i}^i)]^2}
\end{align}
The uniform empirical probability measure on the set $\Theta_{\bs{n}}$ is given by
\begin{equation}\label{eq:gq-measure}
  \nu_{\bs{n}} = \bigotimes_{i=1}^d \nu_{n_i}^i = \frac{1}{\prod_{i=1}^d n_i} \sum_{\bs{k} \leq \bs{n}} \delta_{z_{\bs{k}}} =  \sum_{\bs{k} \leq \bs{n}} \prod_{i=1}^d \frac{1}{n_i} \delta_{z^i_{k_i}}
\end{equation}
Note that iid sampling from $\nu_{\bs{n}}$ is equivalent to sampling uniformly from a tensor-product Gauss quadrature grid.

We will use the notation $\E_{\nu^i_n}$ and $\E_{\nu_{\bsn}}$ to denote expectations under the measures defined in \eqref{eq:marginal-gq-measure} and \eqref{eq:gq-measure}, respectively. I.e.,
\begin{align*}
  \E_{\nu^i_n} f\left(X^i\right) &\triangleq \frac{1}{n} \sum_{k=1}^n f\left(z^i_k\right), &
  \E_{\nu_{\bsn}} f\left(\bsX\right) &\triangleq \frac{1}{\prod_{i=1}^d n_i} \sum_{\bs{1} \leq \bsk \leq \bsn} f\left({\bsz}_{\bsk}\right)
\end{align*}

\subsection{Orthogonal matrices from Gaussian quadrature}

The defining property of the univariate Gauss quadrature rule allows one to conclude that weighted Vandermonde-like matrices formed on the quadrature nodes are orthogonal matrices. Later, we will need the notion of an orthogonal rectangular matrix.
\begin{definition}\label{def:rectangular-orthogonal-matrix}
  An $M \times N$ matrix $\bs{D}$ with $M \geq N$ is \textit{orthogonal} if $\bs{D}^T \bs{D} = \bs{I}_{N \times N}$.
\end{definition}
With $\varphi_k^i$ the degree-$k$ polynomial from the univariate $\rho^i$-orthonormal family, consider the $n_i \times n_i$ Vandermonde-like matrix $\boldsymbol{\Psi^i}$ with entries
\begin{align}\label{eq:psimat-def}
  \left(\Psi^i\right)_{j,k} &= \varphi^i_{k-1}\left(z^i_j\right), & j,k &= 1, \ldots, n_i
\end{align}
where $z^i_j$ are the nodes of the $n_i$-point Gaussian quadrature rule introduced in \eqref{eq:gq-nodes}. We also need a diagonal matrix $\bs{\Sigma^i}$ containing the quadrature weights:
\begin{align}\label{eq:sigmamat-def}
  \left(\Sigma^i\right)_{j,k} &= w^i_j \delta_{j,k}, & j,k &= 1, \ldots, n_i,
\end{align}
with $\delta_{j,k}$ the Kronecker delta. The degree of exactness of each Gaussian quadrature rule implies
\begin{align*}
  \delta_{j,k} = \int_{\Gamma^i} \varphi_{j-1}\left(s\right) \varphi_{k-1}\left(s\right) \rho^i(s) \dx{s} = \sum_{k=1}^{n_i} \varphi_{j-1}\left(z^i_k\right) \varphi_{k-1}\left(z^i_k\right) w^i_k = \left(\bs{\Psi^i}^T \bs{\Sigma^i} \bs{\Psi_i} \right)_{j,k}
\end{align*}
for $j,k = 1, \ldots, n_i$. The fact that the Kronecker delta is equal to the $(j,k)$ element of the matrix $\bs{\Psi^i}^T \bs{\Sigma^i} \bs{\Psi^i}$ indicates that this matrix product is the identity matrix. So we have proven:
\begin{subtheorem}{lemma}\label{lemma:D-orthogonality}
\begin{lemma}[\cite{gautschi_condition_1983}]\label{lemma:D-orthogonal-matrix}
  The $n_i \times n_i$ matrix
  \begin{align}\label{eq:D-univariate}
    \bs{D^i} = \left(\bs{\Sigma^i}\right)^{1/2} \bs{\Psi^i},
  \end{align}
  defined by \eqref{eq:psimat-def} and \eqref{eq:sigmamat-def} is an orthogonal matrix.
\end{lemma}

A straightforward consequence of this is that the $\left(\prod_{i=1}^d n_i\right) \times \left(\prod_{i=1}^d n_i\right)$ matrix that is the tensor product of the $d$ univariate matrices $\bs{D^i}$ is also an orthogonal matrix:
\begin{lemma}[\cite{TangIa}]
  The $\left(\prod_{i=1}^d n_i\right) \times \left(\prod_{i=1}^d n_i\right)$ matrix,
  \begin{align}\label{eq:D-square}
    \bs{D} = \bigotimes_{i=1}^d \bs{D^i},
  \end{align}
  is an orthogonal matrix.
\end{lemma}

In \cite{TangIa}, the matrix $\bs{D}$ was called a ``discrete orthogonal matrix" (DOM). Note that the same properties hold if $\bs{D}$ has more rows than columns. (I.e., the marginal quadrature rule order dominates the marginal polynomial degree of the basis.)
\begin{lemma}
  Let $\bsm, \bsn \in \N^d$ satisfy $\bsm \geq \bsn$.
  For each $i$, define the $m_i \times n_i$ matrix $\bs{\Psi^i}$ with entries
  \begin{align*}
    \left(\Psi^i\right)_{q,k} &= \varphi_{k-1}\left(z^i_q\right), & k &= 1, \ldots, n_i, \hskip 10pt q = 1, \ldots, m_i.
  \end{align*}
  With $\bs{\Sigma^i}$ the $m_i \times m_i$ diagonal weighting matrix given in \eqref{eq:sigmamat-def}, define the $m_i \times n_i$ matrix,
  \begin{align*}
    \bs{D^i} = \left(\bs{\Sigma^i}\right)^{1/2} \bs{\Psi^i}.
  \end{align*}
  Then the $\left(\prod_{i=1}^d m_i\right) \times \left(\prod_{i=1}^d n_i \right)$ matrix
  \begin{align}\label{eq:D-rectangular}
    \bs{D} = \bigotimes_{i=1}^d \bs{D^i}
  \end{align}
  is an orthogonal matrix.
\end{lemma}
The most general version of this statement is that any matrix $\bs{D}$ whose columns are a subset of the appropriate tensor-product space is orthogonal.
\begin{lemma}
  Let $\Lambda \in \N_0^d$ be a finite multi-index set, and let $\bs{n}$ be any multi-index satisfying
  \begin{align*}
    \Lambda \subseteq \Lambda_{\bs{n-1}}^P.
  \end{align*}
  With $N = |\Lambda|$, let $\bs{k}(1), \ldots, \bs{k}(N)$ represent any enumeration of the elements of $\Lambda$. With $M = \prod_{i=1}^d n_i$, let $\bsm(1), \ldots, \bsm(M)$ denote any enumeration of the elements in $\left\{ \bsm \;\; | \;\; \bs{1} \leq \bsm \leq \bsn + \bs{1}\right\}$. Then the $M \times N$ matrix $\bs{D}$ with entries
  \begin{align*}
    (D)_{j,\ell} &= \sqrt{w_{\bsm(j)}} \phi_{\bsk(\ell)}\left(\bsz_{\bsm(j)}\right), & 1 \leq j \leq M,\;\; & 1 \leq \ell \leq N
  \end{align*}
  is an orthogonal matrix.
\end{lemma}
\end{subtheorem}

That $\bs{D}$ is an orthogonal matrix is interesting because it indicates that it is a well-conditioned matrix. If, in addition, the mass of the matrix is equidistributed across all its entries, then Theorem \ref{RIP_bos} implies that it has a small RIP constant. Thus, if equidistribution holds, then it might be possible to use its rows as a discrete candidate set to subsample for a compressive sampling strategy. This was explored in \cite{TangIa} when $\bs{X}$ is a uniform random variable. Here we extend these results to more general cases. We delay the analysis and converence results until Section \ref{sec:analysis}.

\subsection{Compressive sampling via Gaussian quadrature subsampling}\label{sec:gpc-stuff:algorithm}

A simple algorithm can now be presented for compressive sampling recovery of a random function $f(\bs{X})$: we can subsample $M$ rows from the matrix $\bs{D}$ in the previous section and use them to perform sparse recovery. A more detailed algorithm is as follows:
\begin{enumerate}
  \item Given an index set $\Lambda$, find $\bs{n} \in \N^d$ such that $\Lambda \subseteq \Lambda^P_{\bs{n} - \bs{1}}$.
  \item Generate the $n_i$-point Gauss quadrature rule $\left\{z^i_k, w^i_k\right\}_{k=1}^{n_i}$ for $i=1, \ldots, d$. The resulting full tensor-product rule is $\Theta_{\bs{n}} = \left\{\left(\bs{z}_{\bs{k}}, w_{\bs{k}}\right)\right\}_{\bs{1} \leq \bs{k} \leq \bs{n}}$, but it need not be constructed explicitly.
  \item According to the uniform probability law, randomly choose (and construct) $M$ points $\left\{ \left(\bs{y}_m, v_m \right)\right\}_{m=1}^M \subset \Theta_{\bsn}$.
  \item Generate an $M$-row matrix $\bs{D}$ from these $M$ points and index set $\Lambda$, having entries
    \begin{align}\label{eq:general-D-def}
      (D)_{m,n} &= \sqrt{v_m} \phi_{\bsk(n)}\left(\bs{y}_m\right),
    \end{align}
    where $\bsk(1), \ldots, \bsk(N)$ is an enumeration of the elements in $\Lambda$. Also form the $M \times M$ diagonal matrix $\bs{W}$ with entries $(W)_{m,m} = v_m$, and collect the $M$ evaluations of the function $f(\bs{y}_m)$ into the vector $\bs{f}$.
  \item Solve \eqref{eq:weighted-l1minimization} for the coefficients $\bs{c}$.
\end{enumerate}
This procedure amounts to subsampling the product Gaussian quadrature rule in order to perform compressive sampling. One remaining question is how large $M$ should be so that we can guarantee recovery. The authors in \cite{TangIa} show that if $\bsX$ is a uniform random variable, then this strategy requires $M \gtrsim 3^d s$ samples to recover an approximately $s$-sparse vector. We provide a similar analysis for more general random variables by using analysis presented in the next section. Numerical results from the above algorithm are shown in Section \ref{sec:results} for various probability densities for $\bs{X}$.

\section{Analysis of Gaussian quadrature subsampling}\label{sec:analysis}

Here we present analysis of the algorithm shown in Section \ref{sec:gpc-stuff:algorithm}. The essential question is how many samples $M$ are required so that we can guarantee a faithful recovery of some function $f$. We proceed to show this by analyzing the entries of the weighted Vandermonde-like matrix $\bs{D}$. The entries of this matrix can be viewed as non-polynomial functions orthornormal under the discrete measure $\nu_{\bs{n}}$ defined in \eqref{eq:marginal-gq-measure}. Using properties of orthogonal polynomials, we can determine the maximum magnitude of these functions, allowing us to use Theorem \ref{RIP_bos} to determine a sufficient number of samples $M$. Our main result is stated in Theorem \ref{thm:main-sample-count}.

\subsection{Discrete orthonormal systems}

The polynomials $\varphi^i_k$ are orthonormal under the orthogonalizing measure $\rho^i$. In this section we show that specially weighted versions of $\varphi^i_k$ are orthonormal under the discrete measures $\nu^i_n$, defined in \eqref{eq:marginal-gq-measure}. We show similar statements for the multivariate polynomials $\phi_{\bs{k}}$ under the measure $\nu_{\bs{n}}$. These results stem from the accuracy of the Gauss quadrature rules introduced in the previous section: The four parts of Lemma \ref{lemma:general-psi-orthonormality} presented below are restatements of the four parts of Lemma \ref{lemma:D-orthogonality}. We show brief proofs for the first two parts.

\begin{subtheorem}{lemma}\label{lemma:general-psi-orthonormality}
\begin{lemma}\label{lemma:psi-univariate-orthonormality}
  Let $n \in \N$ be fixed. Then the $n$ functions
  \begin{align}\label{eq:psi-univariate-def}
    \left\{ \psi^i_{k,n}\left(x^i\right) \right\}_{k=0}^{n-1} \triangleq \left\{ \sqrt{n \lambda^i_n \left(x^i\right)} \varphi_k^i\left(x^i\right) \right\}_{k=0}^{n-1}
  \end{align}
  are orthonormal under the probability measure $\nu^i_n$ defined in \eqref{eq:marginal-gq-measure}.
\end{lemma}
\begin{proof}
  The result follows by direct calculation and use of the exactness of the Gauss quadrature rule, and is essentially equivalent to Lemma \ref{lemma:D-orthogonal-matrix}. For $j, k < n$ we have:
  \begin{align*}
    \E_{\nu^i_n} \left[ \psi^i_{k,n}\left(X^i\right) \psi^i_{j,n}\left(X^i\right) \right] &\stackrel{\eqref{eq:marginal-gq-measure}}{=} \frac{1}{n} \sum_{p=0}^{n-1} \psi^i_{j,n}\left(z^i_p\right) \psi^i_{k,n}\left(z^i_p\right) \\
                                                                                  &\stackrel{\eqref{eq:psi-univariate-def}}{=} \frac{1}{n} \sum_{p=0}^{n-1} n \lambda^i_n\left(z^i_p\right) \varphi_j^i\left(z^i_p\right) \varphi_k^i\left(z^i_p\right) \\
                                                                                  &\stackrel{\eqref{eq:quadratureweight}}{=} \sum_{p=0}^{n-1} w^i_p \varphi_j^i\left(z^i_p\right) \varphi_k^i\left(z^i_p\right) \\
                                                                                  &\stackrel{\eqref{eq:gq-exactness}}{=} \int_{\Gamma_i} \varphi^i_j\left(x^i\right) \varphi^i_k\left(x^i\right) \rho^i\left(x^i\right) \dx{x^i} \\
                                                                                  &\stackrel{\eqref{eq:phi-orthonormality}}{=} \delta_{j,k}.
  \end{align*}
\end{proof}
Note the factor of $\sqrt{n}$ in \eqref{eq:psi-univariate-def}, and the fact that $\psi^i_{k,n}$ depends on $n$. The result above generalizes in a straightforward manner to the multivariate case.
\begin{lemma}\label{lemma:psi-orthonormality}
  Let $\bs{n} \in \N^d$ be fixed. Then the $\prod_{i=1}^d n_i$ functions
  \begin{align}\label{eq:psi-def}
    \left\{ \psi_{\bs{k},\bs{n}}\left(z\right) \right\}_{\bsk \in \Lambda^P_{\bs{n}-\bs{1}}} \triangleq \left\{\prod_{i=1}^d \psi^i_{k_i, n_i}\left(z^i\right)\right\}_{\bsk \in \Lambda^P_{\bs{n}-\bs{1}}}
  \end{align}
  are orthonormal under the probability measure $\nu_{\bs{n}}$ defined in \eqref{eq:gq-measure}.
\end{lemma}
\begin{proof}
  We again use a direct calculation. For $\bs{0} \leq \bs{j}, \bs{k} < \bs{n}$ we have:
  \begin{align*}
    \E_{\nu_{\bs{n}}} \left[ \psi_{\bs{j}}\left(\bsX\right) \psi_{\bs{k}}\left(\bsX\right) \right] &= \E_{\nu^1_{n_1} \otimes \cdots \otimes \nu^d_{n_d}} \left[ \prod_{i=1}^d \psi^i_{j_i}\left(Z^i\right) \psi^i_{k_i} \left( Z^i\right) \right]  \\
                                                                                             &= \prod_{i=1}^d \E_{\nu^i_{n_i}} \left[ \prod_{i=1}^d \psi^i_{j_i}\left(Z^i\right) \psi^i_{k_i} \left( Z^i\right) \right] \\
                                                                                             &\stackrel{\textrm{Lemma \ref{lemma:psi-univariate-orthonormality}}}{=} \prod_{i=1}^d \delta_{j_i, k_i} = \delta_{\bs{j}, \bs{k}},
  \end{align*}
  In the second equality above we have used the fact that a random variable associated to the measure $\nu_{\bs{n}}$ has independent components $Z^i$ with marginal distributions given by $\nu^i_{n_i}$.
\end{proof}
Of course, one can use a quadrature rule of higher-degree accuracy than required and still retain orthogonality of the resulting matrix.
\begin{lemma}
  Let $\bsm, \bsn \in \N^d$ satisfy $\bsm \geq \bsn$. The $\prod_{i=1}^d n_i$ functions
  \begin{align*}
    \left\{ \psi_{\bs{k},\bs{m}}\left(z\right) \right\}_{\bsk \in \Lambda^P_{\bsn - \bs{1}}},
  \end{align*}
  are orthonormal under the probability measure $\nu_{\bsm}$.
\end{lemma}
Finally, any non-tensor-product polynomial space of finite dimension can be encapsulated in a tensor-product space, and the quadrature rule associated to the tensor-product space exactly integrates elements from the original space.
\begin{lemma}
  Given a finite index set $\Lambda$, let $\bs{n} \in \N^d$ be such that $\Lambda \subseteq \Lambda_{\bsn - \bs{1}}^P$. With $N = |\Lambda|$, then the $N$ functions
  \begin{align*}
    \left\{ \psi_{\bs{k},\bs{n}}\left(z\right) \right\}_{\bsk \in \Lambda},
  \end{align*}
  are orthonormal under the probability measure $\nu_{\bsn}$.
\end{lemma}
\end{subtheorem}

\subsection{Bounded discrete orthonormal systems}
We have established in the previous section that the collection of functions $\left\{\psi_{\bs{k}, \bs{n}} \right\}_{\bs{k} \in \Lambda}$ are $\nu_{\bs{n}}$-orthonormal for any $\Lambda \subseteq \Lambda^P_{\bs{n}-\bs{1}}$. The recovery result in Theorem \ref{RIP_bos} then guarantees RIP properties associated to iid sampling strategies based on sup-norm bounds of these functions, to which we now turn.

Consider the scalar random variable $X^i$ with $n \in \N$ fixed. The $\psi^i_{k,n}$ are $\nu^i_n$-orthonormal, and so we must determine the bound
\begin{align*}
  L_i(n) \triangleq \max_{0 \leq k \leq n-1} \sup_{x^i \in \supp \nu^i_n} \left|\psi_{k,n}\left(x^i\right)\right|^2 = \max_{0 \leq k \leq n-1} \max_{1 \leq j \leq n} \left|\psi^i_{k,n}\left(z^i_j\right)\right|^2,
\end{align*}
where $z^i_j$ are the $n$-point Gaussian quadrature nodes. The $\psi_{k,n}$ functions are polynomials weighted by a Christoffel function. Much is known about the behavior of these functions, and can be used to estimate the bound $L$. The following lemmas establish the behavior of $L_i(n)$ for general distributions of $X^i$. The proofs of all the below Lemmas are in the Appendix.

When the polynomial family $\varphi_k$ that defines the weighted functions $\psi_{k,n}$ corresponds to almost any Jacobi polynomial family, then the bounding constant $L_i(n)$ is \textit{independent} of $n$, and in this case the $\psi_{k,n}$ are a uniformly bounded system for $k \leq n$.
\begin{subtheorem}{lemma}\label{lemma:bounds}
\begin{lemma}\label{lemma:jacobi-bound}
  Let $X^i \sim B\left(\gamma+1, \delta+1\right)$ be a univariate Beta-distributed random variable with shape parameters $\gamma, \delta \geq -\frac{1}{2}$ on the domain $X^i = [-1,1]$. Thus, the polynomials $\varphi^i_k$ are Jacobi polynomials with parameters $\delta, \gamma$. Then
    \begin{align*}
      L_i(n) \leq C
    \end{align*}
    The constant $C$ is uniform in $n$, $C = C(\gamma, \delta)$.
\end{lemma}
We show in the left-hand plot of Figure \ref{fig:univariate-sup-bounds} the evolution of $L_i(n)$ as a function of the quadrature node size $n$, along with its dependence on symmetric parameters $\gamma = \delta$.

The strict $n$-independent bound obtained above for Jacobi polynomial families is, unfortunately, not true for polynomials orthogonal with respect to exponential weights on unbounded domains. Instead, the bound depends on the degree $n$. However, the dependence is relatively mild.
\begin{lemma}\label{lemma:full-exponential-bound}
  Let $X^i$ be a random variable whose density $\rho^i$ is exponential on the real line:
  \begin{align*}
    \rho^i(x) &\propto \exp(-|x|^\alpha), & x &\in \R
  \end{align*}
  for some $\alpha > \frac{3}{2}$. Then the weighted polynomials $\psi^i_{k,n}$ satisfy
    \begin{align*}
      L_i(n) \leq C n^{2/3},
    \end{align*}
    where the constant $C$ is uniform in $n$, i.e., $C = C(\alpha)$.
\end{lemma}
Note that the above case covers $\alpha = 2$, corresponding to a normally-distributed $X^i$ and a gPC basis of Hermite polynomials. The one-sided exponential bound is similar; we state it separately because its proof in the Appendix requires a different set of results.
\begin{lemma}\label{lemma:half-exponential-bound}
  Let $X^i$ be a random variable whose density $\rho^i$ is exponential on the half real line:
  \begin{align*}
    \rho^i(x) &\propto \exp(-|x|^\alpha), & x &\geq 0
  \end{align*}
  for some $\alpha > \frac{3}{4}$. Then the weighted polynomials $\psi^i_{k,n}$ satisfy
    \begin{align*}
      L_i(n) \leq C n^{2/3},
    \end{align*}
    where the constant $C$ is uniform in $n$, i.e., $C = C(\alpha)$.
\end{lemma}
\end{subtheorem}
\begin{remark}
  We expect the conclusion of Lemma \ref{lemma:half-exponential-bound} to be valid for the more general weight $\rho^i \propto x^{\mu} \exp(-|x|^\alpha)$ for some $\mu \geq -\frac{1}{2}$. This would require some estimates on zeros of the associated orthogonal polynomials along with behavior of the associated Christoffel function. These estimates are essentially present in \cite{levin_orthogonal_2005,levin_orthogonal_2006}. Since this would necessitate a more technical analysis with dubious payoff, we do not pursue this here.
\end{remark}
Of special note is the result of Lemma \ref{lemma:half-exponential-bound} with $\alpha = 1$, corresponding to an exponential random variable $X^i$ and a gPC basis of Laguerre polynomials.

The behavior of the bounds established above are shown in the right-hand plot of Figure \ref{fig:univariate-sup-bounds} for Hermite polynomials (Lemma \ref{lemma:full-exponential-bound} with $\alpha = 2$) and Laguerre polynomials (Lemma \ref{lemma:half-exponential-bound} with $\alpha = 1$).
\begin{figure}
  \begin{center}
    \includegraphics[width=\textwidth]{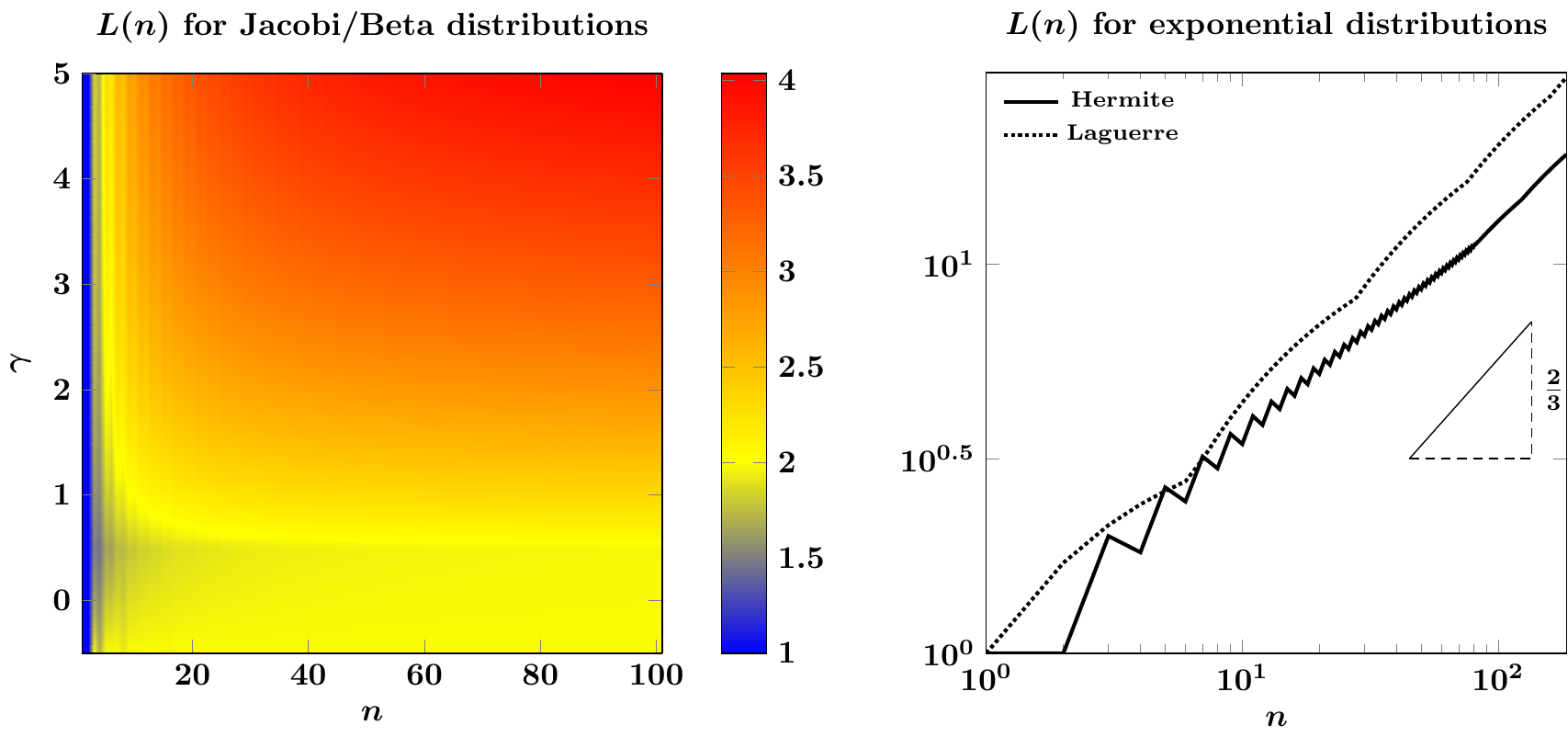}
  \end{center}
  \caption{Left: Bound $L$ for Jacobi polynomials with symmetric parameters $\gamma = \delta$. Right: Bound for two exponential-type densities: Hermite polynomials with $\rho = \exp(-x^2)$ on $\R$, and Laguerre polynomials with $\rho = \exp(-x)$ on $[0, \infty)$.}\label{fig:univariate-sup-bounds}
\end{figure}

\subsection{Multivariate expansions}
This section contains our main results that utilize the univariate supremum bounds established in the previous section. Given $\bs{n} \in \N^d$, we define the product bound $L(\bs{n})$:
\begin{align}\label{eq:multivariate-L}
  L(\bs{n}) = \prod_{i=1}^d L_i(n_i)
\end{align}

Our procedure runs into the familiar curse of dimensionality by requiring the sample count to dominate $s$ times the product of all $d$ of the $L_i$ factors.
\begin{theorem}\label{thm:main-sample-count}
  Let $\Lambda$ be a finite index set with size $N$. Let $\bs{n}$ be the smallest multi-index such that $\Lambda \subseteq \Lambda^P_{\bs{n-1}}$. (This defines $\bsn$ uniquely.)
  Choose $M$ samples randomly without replacement from the measure $\nu_{\bs{n}}$. (This is the algorithm presented in Section \ref{sec:gpc-stuff:algorithm}.) Assume $M$ satisfies
  \begin{align}\label{eq:sample-count-criterion}
    M \geq L\left(\bs{n}\right) C_1 s \log^3(s) \log\left(N\right),
  \end{align}
  where $C_1$ is a universal constant. Under these conditions, then for any $\bs{c} \in \R^{N}$, let $\bs{c}^\sharp$ be the solution obtained by the weighted $\ell^1$ optimization problem defined by \eqref{eq:weighted-l1minimization}. Then,
  \begin{align*}
    \mathrm{Pr} \left[ \left\| \bs{c} - \bs{c}^\sharp\right\|_2 \leq \frac{C_2 \sigma_{s,1}\left(\bs{c}\right)}{\sqrt{s}} \right] \leq 1 - N^{-\gamma \log^3(s)},
  \end{align*}
  where $C_2$ and $\gamma$ are universal constants. The individual factors $L_i(n_i)$ in \eqref{eq:multivariate-L} that comprise $L(\bs{n})$ in \eqref{eq:sample-count-criterion} are bounded by the three parts of Lemma \ref{lemma:bounds} when the components of $\bs{X}$ have the appropriate distributions.
\end{theorem}
The proof of the above Theorem, utilizing Lemmas \ref{lemma:bounds}, is given in Appendix \ref{sec:main-theorem}. The above theorem quantifies the size of $M$ so that the subsampling algorithm in Section \ref{sec:gpc-stuff:algorithm} converges with high probability. In particular, we frame this result slightly differently compared to the conclusion of Theorem \ref{RIP_bos}: we make an explicit choice for the RIP constant $\delta$ so that recoverability is guaranteed (cf. Theorem \ref{thm:l1-l0-connection}).

Lemmas \ref{lemma:bounds} cover many of the standard univariate distributions for $X^i$, in both the bounded and unbounded cases. Thus, the theorem above applies to very general cases of a random variable $\bs{X}$ with independent components. We point out some special cases of our result:
\begin{itemize}
  \item If $\rho$ is uniform over $\Gamma = [-1,1]^d$, so that the $\varphi_{\bs{k}}$ are tensor-product Legendre polynomials, then $L(\bs{n})$ satisfies
    \begin{align*}
      L(\bs{n}) \leq C^d,
    \end{align*}
    with $C$ the univariate bound in Lemma \ref{lemma:jacobi-bound} with $\gamma = \delta = 1$. We notice that the above constant depends exponentially on the dimension $d.$ The authors in \cite{TangIa} show that this constant is essentially $C = 3$. Results from Figure \ref{fig:univariate-sup-bounds} suggest that a sharper result would be $C = 2$, although this is not proven.
  \item If $\rho$ corresponds to a standard Gaussian density function over $\mathbb{R}^d$, so that the $\varphi_{\bs{k}}$ are tensor-product Hermite polynomials, then each univariate $L_i(n_i)$ satisfies an $n_i^{2/3}$ bound, so that
    \begin{subequations}\label{eq:hermite-bounds}
    \begin{align}\label{eq:quadratric-bound-hermite}
      L(\bs{n}) &\leq C^{2d/3} {\bs{n}}^{\bs{2/3}}, & \bs{2/3} &= \left(2/3, 2/3, \ldots, 2/3\right) \in \mathbb{R}^d,
    \end{align}
    with $C$ the constant given in Lemma \ref{lemma:full-exponential-bound}. If the maximum polynomial degree $n_i - 1 \triangleq n-1$ is uniform for all dimensions, then we have
    \begin{align}\label{eq:quadratic-bound-hermite-isotropic}
      L(\bs{n}) &\leq (C n)^{2d/3},
    \end{align}
    This result unfortunately exhibits not only exponential dependence on the parametric $d$, but also algebraic dependence on the maximum polynomial degree $n$. Nevertheless, this bound on the supremum of the Gauss quadrature-weighted polynomials is sharp. However, the results in Figure \ref{fig:univariate-sup-bounds} show that for polynomial degree $n-1 = 9$, we empirically observe $C n^{2/3} \lesssim 4$, so that for a degree-9 polynomial approximation, the requirement \eqref{eq:sample-count-criterion} states
    \begin{align}\label{eq:sample-count-hermite}
      M &\geq 4^d s \log^3(s) \log\left(N\right), & (n \leq 10)
    \end{align}
  \end{subequations}
    We note that in high dimensions one is more likely to use low-degree approximations (small $n$) so that in high dimensions this requirement is comparable to the bounded-case sample count criterion.
  \item If $\rho$ is given by $\rho(\bs{x}) \propto \exp \left(-\left\|\bs{z}\right\|_1\right)$ for $\bs{z} \in [0, \infty)^d$, then we also obtain the set of bounds \eqref{eq:hermite-bounds} under the same conditions.
\end{itemize}
Finally, we remark that, once $\bs{n}$ is identified as the multi-index identifying maximum the polynomial degree in each dimension, one may choose to subsample from the measure $\nu_{{\bs{m}}}$, where ${\bs{m}} \geq \bs{n}$. When $\bs{X}$ is a bounded random variable, this would produce the same bound \eqref{eq:sample-count-criterion} with $\bs{n}$, even though one subsamples from $\nu_{{\bs{m}}}$. However, if $\bs{X}$ has an exponential density (either one-sided or two-sided exponential), our analysis suggests that such a strategy would have a bound $L$ that behaves like $L \sim \bs{m}^{\bs{2/3}} > \bs{n}^{\bs{2/3}}$. The penalty for this unbounded case makes sense: using Gauss quadrature rules with ${\bs{m}} > \bs{n}$ results in possible sampling of points that lie in regions where degree-$\bs{n}$ polynomials weighted by $\lambda_{{\bs{m}}}$ are decaying quickly to 0, making recovery by sampling in these regions difficult. It is likely that one can improve the estimates in \cite{jakeman_2015} to result in a tighter bound, but we still expect this bound to be greater than $\bs{n}^{\bs{2/3}}$.

\section{Numerical examples}\label{sec:results}

We now provide with some numerical examples to test the theoretical
findings and the convergence properties of the sparse recovery approach, with randomly chosen Gaussian grid.
For the implementation of the $\ell_1$-minimization,
we employ available tools such as Spectral Projected Gradient algorithm
(SPGL1) from \cite{vanFrie} that was implemented in the MATLAB
package SPGL1 \cite{vanFrie2}.

\subsection{Polynomial function recovery}

We first assume the  target function has a polynomial form and choose a
sparsity level $s$ and then fix $s$ coefficients of the polynomial
while keeping the rest of the coefficients zero. The
values of the $s$ non-zero coefficients are drawn iid from a standard normal
distribution. This procedure produces
target coefficients  that we seek to recover using the
$\ell_1$-minimization algorithms.

In what follows, we use the following terms in figures to describe our recovery procedures:
\begin{itemize}
  \item \textit{Random} -- sampling iid from the orthogonality measure $\rho$ and solving the unweighted $\ell_1$ optimization problem \eqref{eq:unweighted-l1minimization}.
  \item \textit{PreChebyshev} -- sampling iid from the Chebyshev measure, i.e., the measure with probability density
    \begin{align*}
      v(\bs{x}) = \frac{1}{\pi^d \prod_{i=1}^d \sqrt{1 - x_i^2}},
    \end{align*}
    and solving the preconditioned optimization problem \eqref{eq:weighted-l1minimization} using the weights \eqref{eq:chebyshev-weights}.
  \item \textit{Gaussian} -- subsampling from a Gaussian quadrature grid using the Gaussian quadrature weights to solve the preconditioned problem \eqref{eq:weighted-l1minimization}. (This is the method proposed in this paper.)
  \item \textit{Chebyshev} (resp. \textit{Uniform}) -- sampling iid from the Chebyshev (resp. uniform) measure and solving the unweighted optimization problem \eqref{eq:unweighted-l1minimization}.
\end{itemize}

We recall our notation: $n_i$ is the number of Gaussian quadrature points in dimension $i$, with $n_i - 1$ being the maximum polynomial degree in dimension $i$ from the index set $\Lambda$. Thus, choosing $\bs{n}$ and $d$ defines the total degree space $\Lambda$. We define the size of $\Lambda$ to be $N$, coinciding with the number of columns in the matrix $\bs{D}$ that is input to the $\ell_1$ optimization problem \eqref{eq:weighted-l1minimization}. The number of samples we use is $M$, and is the number of points subsampled from the tensor-product Gaussian quadrature grid, coinciding with the number of rows of the matrix $\bs{D}$.

\subsubsection{Uniform measure and Legendre polynomials}

The first test is the recovery of sparse Legendre polynomials, with the index set $\Lambda$ corresponding to the two dimensional
total degree space $\mathrm{T}_{n-1}^2$ and $\mathrm{T}_{n-1}^{10}$, respectively. We note that in this case our method coincides with the method in \cite{TangIa}. We examine the frequency of successful recoveries when the number of sample points is fixed at $M = 85$. This is accomplished by conducting 500 trials of the algorithms and counting the successful ones. A recovery is considered successful when the
resulting coefficient vector $\mathbf{c}$ satisfies $\|\mathbf{c}-\hat{\mathbf{c}}\|\leq 10^{-3}$. In the left-hand plot of Figure \ref{fig:legendre_recovery}, we show the recovery rate for sparse Legendre polynomials(with $d=2$, $n_1=n_2=11$, and thus $N=66$) as a function of sparsity level $s$. In the right-hand plot of Figure \ref{fig:legendre_recovery}, we show the recovery rate for sparse Legendre polynomials(with $d=10$, $\bs{n}=\bs{4}$, and $N=286$) as a function of sparsity level $s$.
We have also tested preconditioned recovery results with MC Chebyshev samples. In the cases we have tested, the Gaussian subsampling method works as well or better than the other methods, both for low and high dimensions.
\begin{figure}[t]
\begin{center}
\includegraphics[width=0.49\textwidth]{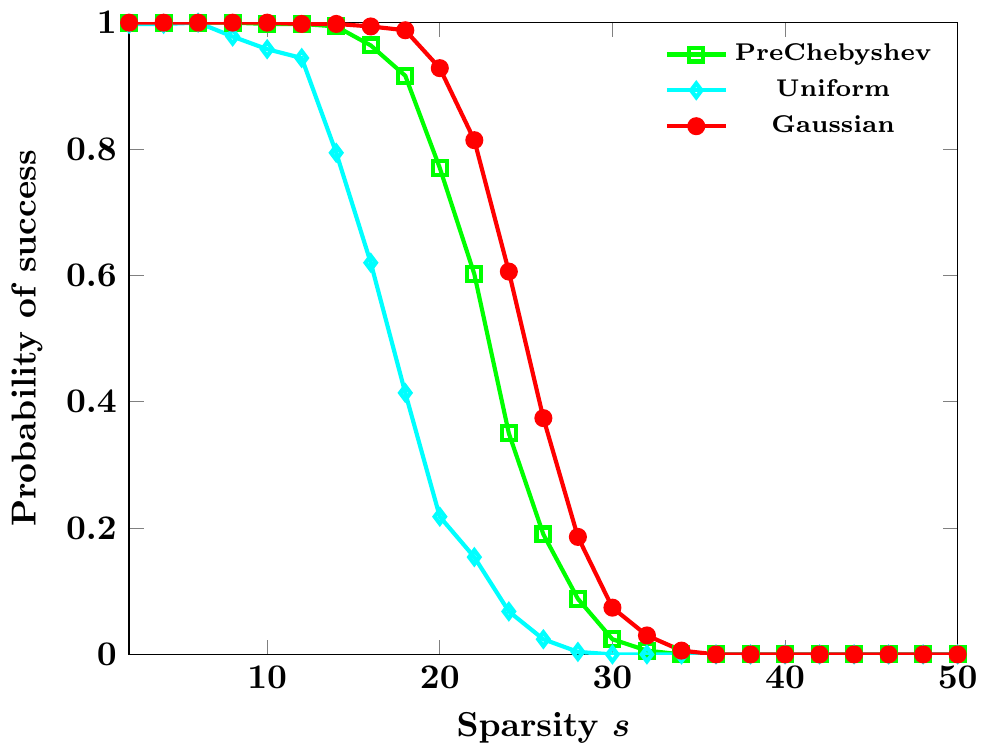}
\includegraphics[width=0.49\textwidth]{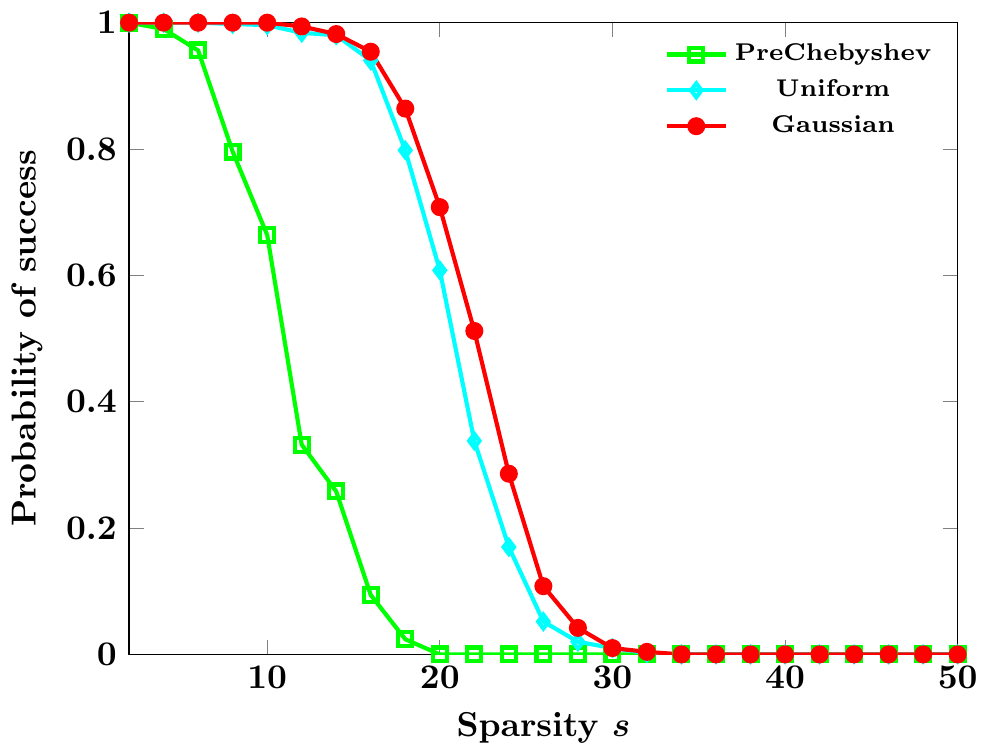}
\end{center}
\caption{Recovery probability with respect to sparsity $s$ of Legendre polynomials with fixed sample number $M=85$. Left: $d=2$, $\bs{n}=\bs{21}$ ($N=231$). Right: $d=10$, $\bs{n}=\bs{4}$ ($N=286$).}\label{fig:legendre_recovery}
\end{figure}

\subsubsection{Chebyshev measure and Chebyshev polynomials}

The second test is the recovery of sparse Chebeshev polynomials, with the index set $\Lambda$ corresponding to the two dimensional
total degree space $\mathrm{T}_{n-1}^2$ and $\mathrm{T}_{n-1}^{10}$, respectively. We examine the probability of successful recoveries when the number of sample points is fixed at $M = 85$. In the left-hand plot of Figure \ref{fig:Chebyshev_recovery},
we show the recovery rate for sparse Chebyshev polynomials(with $d=2$, $\bs{n}=\bs{11}$, and $N=66$) as a function of sparsity level $s$. In this low-dimensional case, the results are similar to recovery of Legendre polynomials in Figure \ref{fig:legendre_recovery}.
In the right-hand plot of Figure \ref{fig:Chebyshev_recovery}, we show the recovery rate for sparse Chebyshev polynomials(with $d=10$, $\bs{n}=4$, and $N=286$) as a function of sparsity level $s$. Again we see that the subsampling Gaussian quadrature case performs better than other methods, although the improvement in the high-dimensional cases is more minor.

\begin{figure}[h]
\begin{center}
\includegraphics[width=0.49\textwidth]{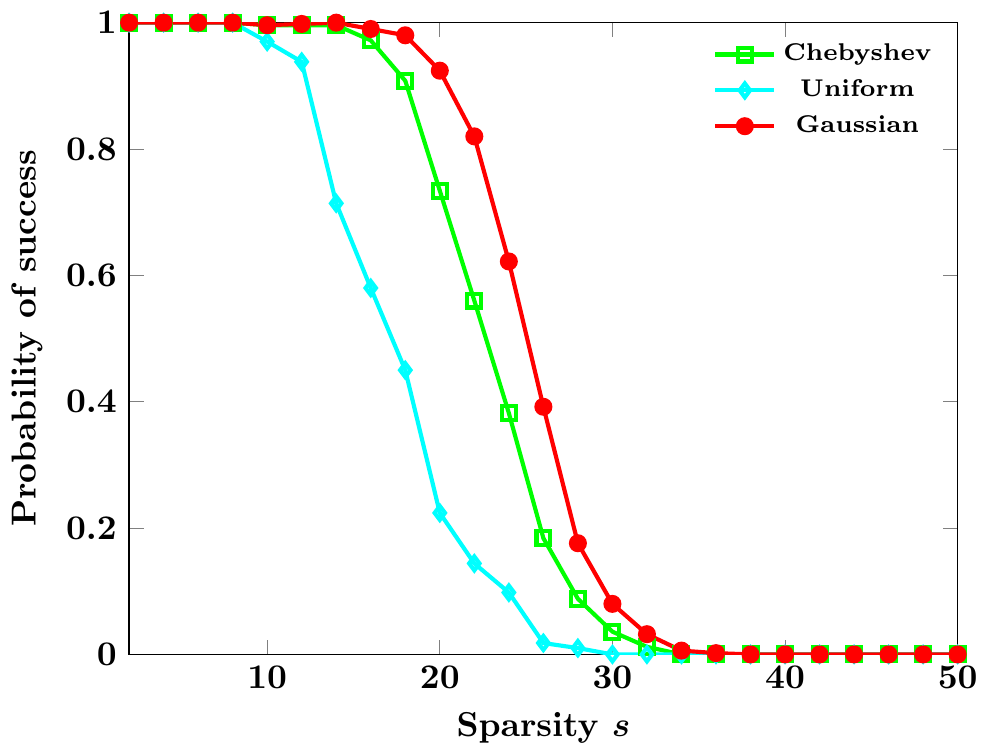}
\includegraphics[width=0.49\textwidth]{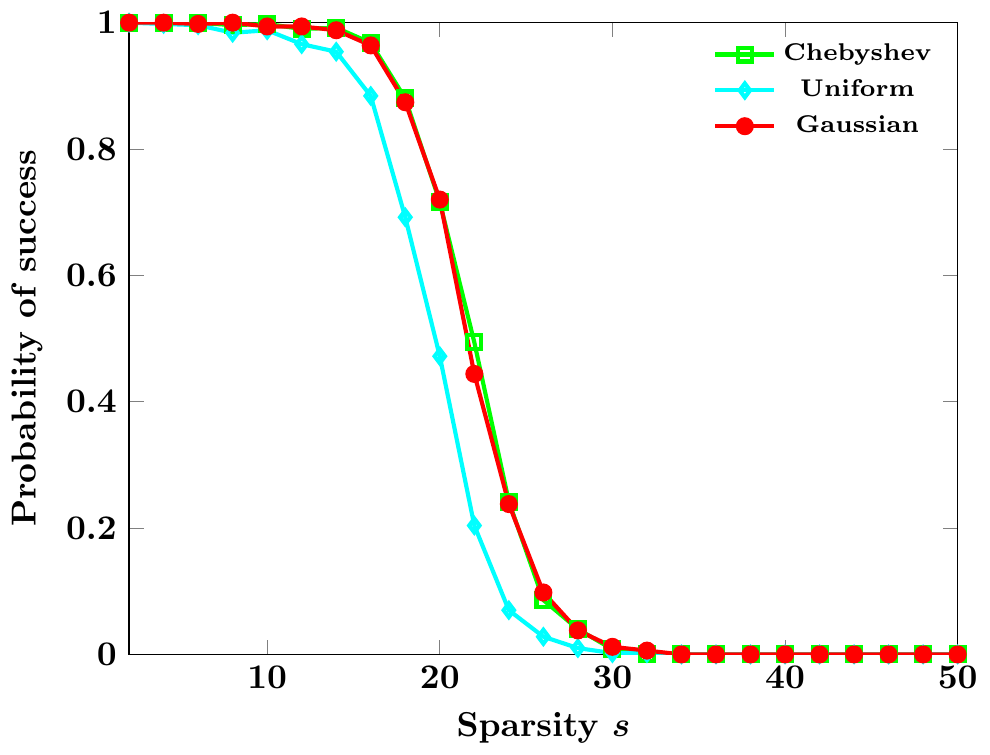}
\end{center}
\caption{Recovery probability with respect to sparsity $s$ of Chebyshev polynomials with fixed sample number $M=85$. Left: $d=2$, $\bs{n}=\bs{21}$ ($N=231$). Right: $d=10$, $\bs{n}=\bs{4}$ ($N=286$).}
\label{fig:Chebyshev_recovery}
\end{figure}

\subsubsection{Gaussian measure with Hermite polynomials}

In Fig.\ref{fig:Hermite_recovery}, we report the numerical results for the Gaussian measure with Hermite polynomials approximation.
We examine the maximum coefficient error, $\|\mathbf{c}-\hat{\mathbf{c}}\|_{\infty}$, as we increase the number of sample points. In the left-hand plot of Figure \ref{fig:Hermite_recovery}, we show the convergence rate for sparse Hermite polynomials (with $d=2$, $\bs{n}=\bs{21}$, and $N=231$) as a function of number of sample points. In the right-hand plot of Figure \ref{fig:Hermite_recovery}, we show the convergence rate for sparse Hermite polynomials(with $d=10$, $\bs{n}=\bs{4}$, and $N=286$) as a function of number of sample points. The sparsity in both cases is $s=5.$ Although subsampling from a tensor-product grid works well in low dimensions, in the high-dimensional $d=10$ case, we see that sampling according to the orthogonality measure produces better results. This is consistent with earlier observations \cite{XiuL1}.

\begin{figure}[h]
\begin{center}
\includegraphics[width=0.49\textwidth]{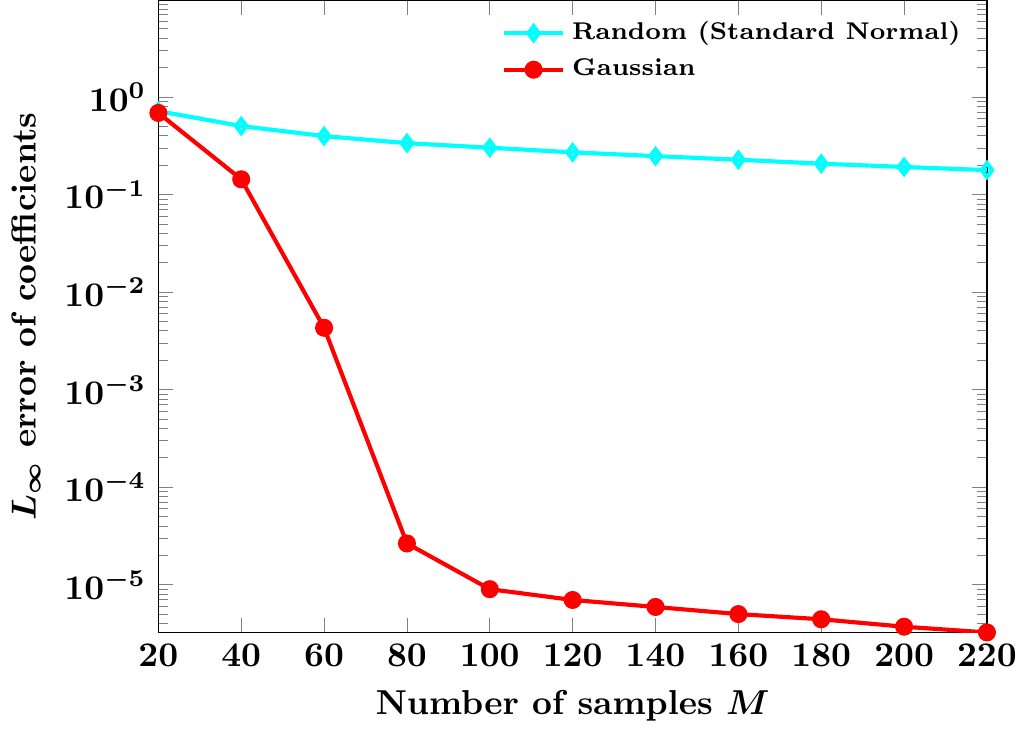}
\includegraphics[width=0.49\textwidth]{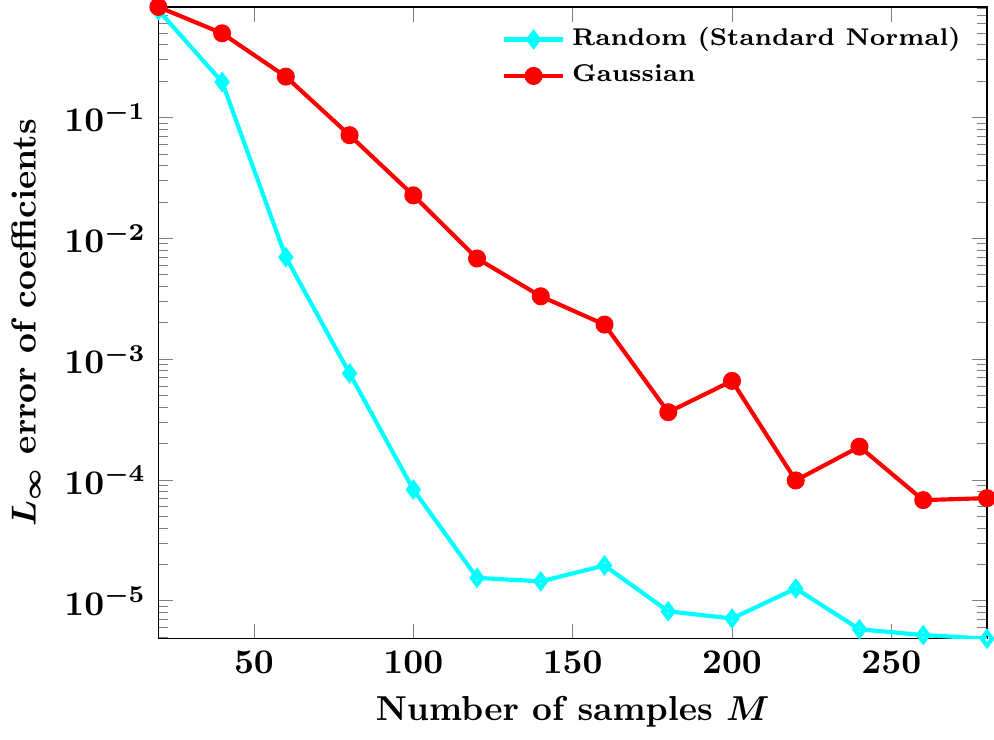}
\end{center}
\caption{Recovery error with respect to number of sample points of Hermite polynomials. Left: $d=2$, $\bs{n}=\bs{21}$, and $N=231$. Right: $d=10$, $\bs{n}=\bs{4}$, and $N=286$. For both plots, the error shown is the average over 500 trials.}\label{fig:Hermite_recovery}
\end{figure}

\subsection{Recovery of analytical functions}
In general, functions do not have a finite representation in the orthogonal polynomials, but instead have ``approximately" sparse representations. Here, we consider a few functions of this form.

We report the numerical error with Legendre polynomials for the underlying high-degree monomial function $f(\mathbf{x})=x_1^{10}x_2^{10}$ and high-dimensional Generalized Rosenbrock $f(\mathbf{x})=\sum\limits_{i=1}^{10}(1-x_i)^2+100(x_{i+1}-x_i^2)^2$. We attempt to recover a sparse representation of these functions in a Legendre polynomial basis. The left-hand plot of Figure~\ref{fig:Legendre_function} shows recovery of the monomial function (high degree, low dimension), and the right-hand plot shows recovery of the Rosenbrock function (low degree, high dimension). In both of these cases, subsampling from a Gaussian quadrature grid produces superior results when compared against standard alternatives.

\begin{figure}[h]
\begin{center}
\includegraphics[width=0.49\textwidth]{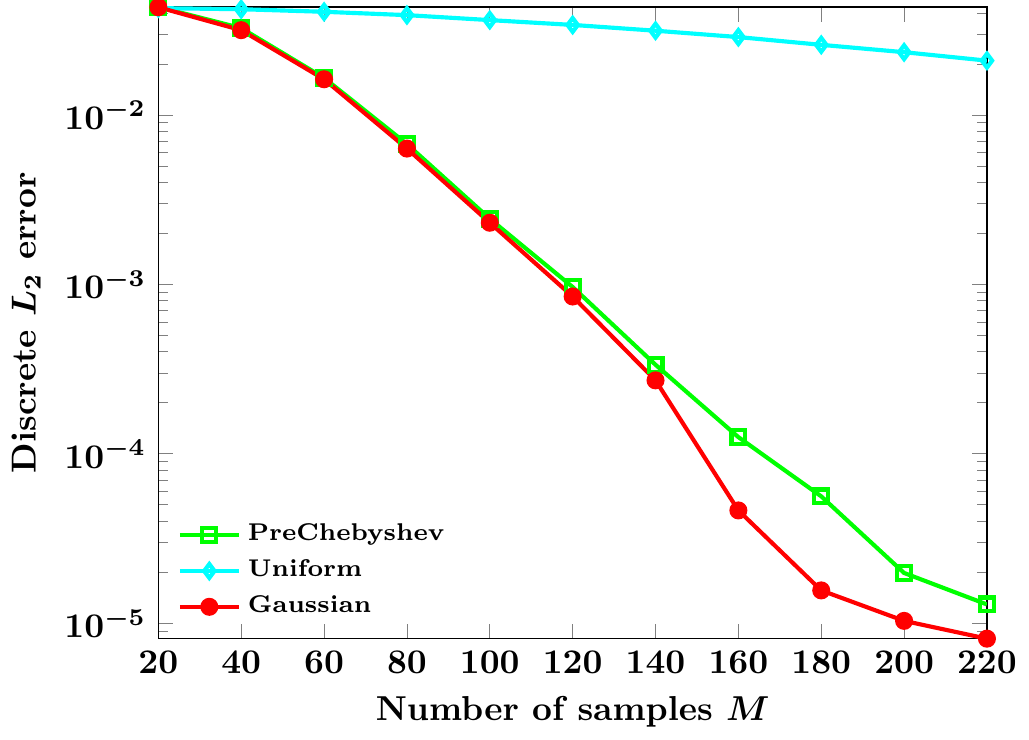}
\includegraphics[width=0.49\textwidth]{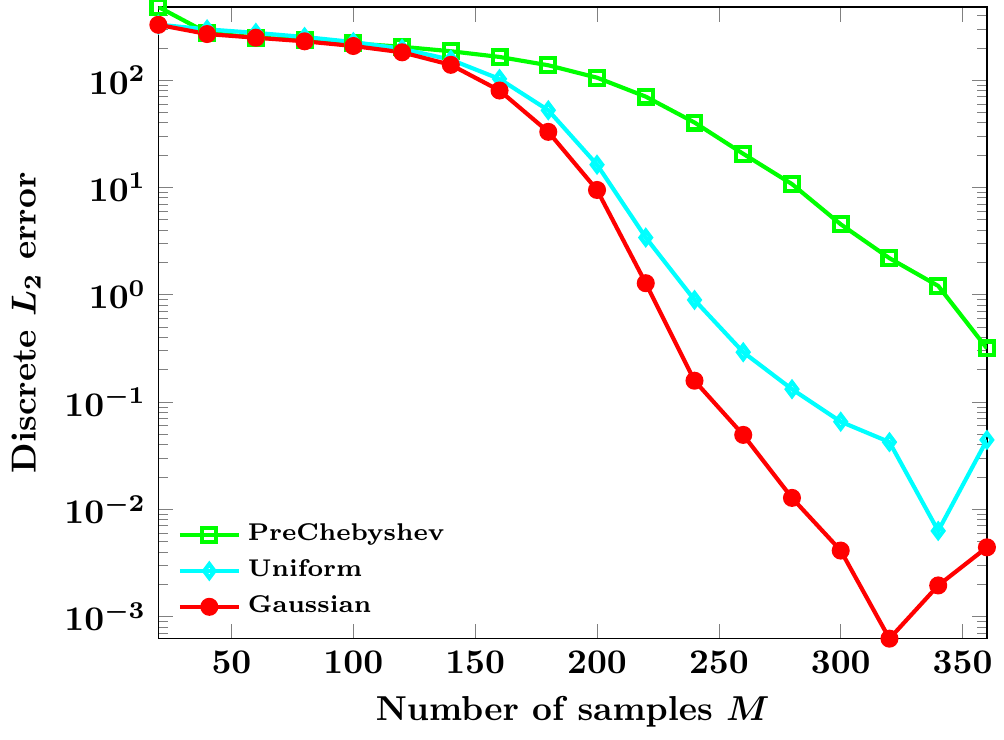}
\end{center}
\caption{Convergence rate against number of sample points. Left:High-degree monomial function($d=2,\bs{n}=\bs{21}$). Right: Generalized Rosenbrock ($d=10,\bs{n}=\bs{5}$). For both plots, the error is averaged over 500 trials.}\label{fig:Legendre_function}
\end{figure}

We run a similar test using Hermite polynomials for two different functions in two dimensions: $f_1(\mathbf{x})=\textmd{2}^{-0.2x_1^2-0.2x_2^2}$ and $f_2(\mathbf{x})=\textmd{e}^{-0.6x_1-0.6x_2}$. The recovery results are shown in Figure \ref{fig:Hermite_function}, where we see that in these low-dimensional cases, it is not always clear that subsampling from a Gaussian quadrature grid produces better results than sampling from the orthogonality measure. However, this does appear to be true if one can afford to take a larger sample count $M$.
\begin{figure}[h]
\begin{center}
\includegraphics[width=0.49\textwidth]{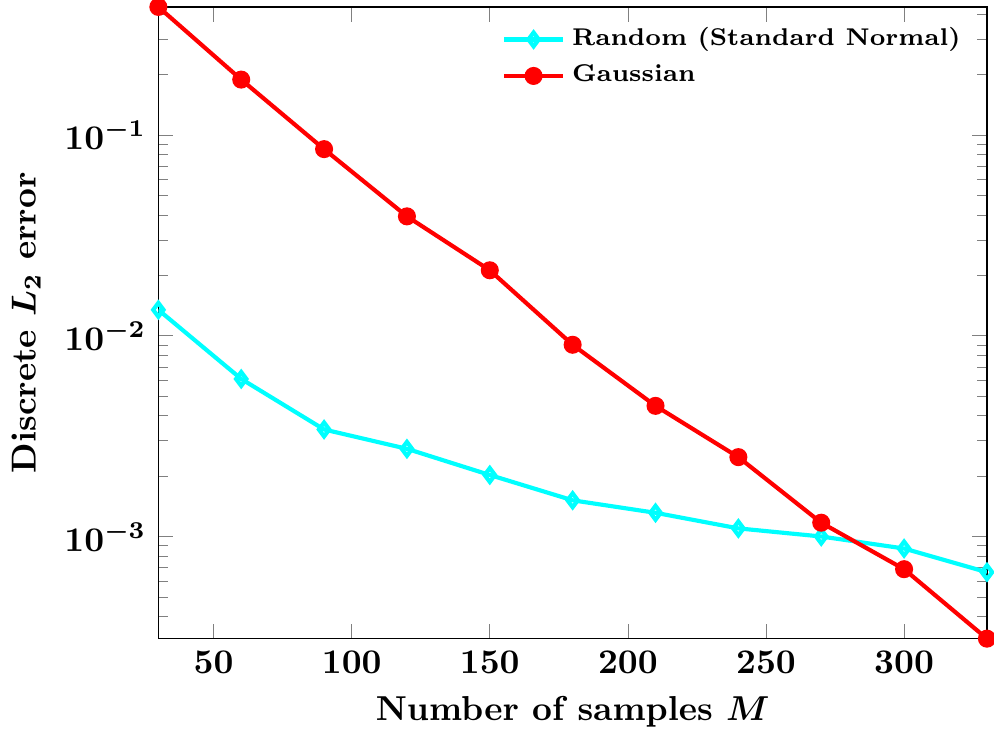}
\includegraphics[width=0.49\textwidth]{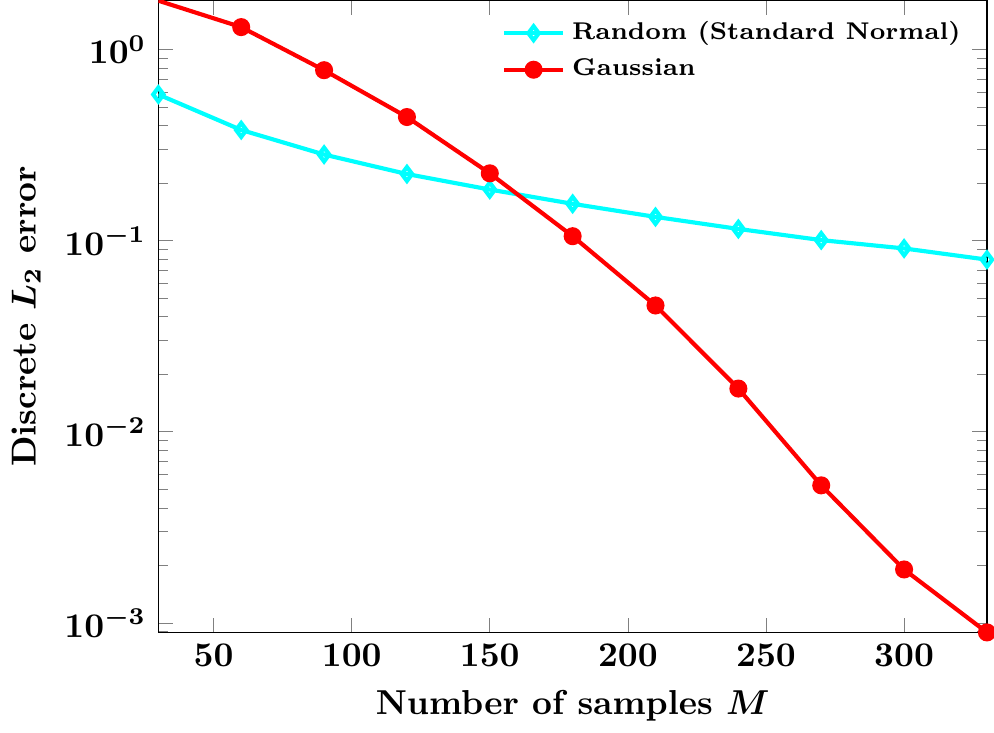}
\end{center}
\caption{Convergence rate against number of sample points. Left:$f_1(\mathbf{x})=2^{-0.2x^2-0.2x_2^2}.$($d=2,\bs{n}=\bs{26}$). Right: $f_2(\mathbf{x})=e^{-0.6x_1-0.6x_2}.$($d=2,\bs{n}=\bs{25}$). In both plots, the error is averaged over 500 trials.}\label{fig:Hermite_function}
\end{figure}

\subsection{A simple ODE with random inputs}
We consider a simple random ODE problem with Gaussian random input:
\begin{align}
\frac{du}{dt}=-k(X)u, \quad u(0)=1,
\label{ode}
\end{align}
where $k(X)$ is a function of a Gaussian random variable $X$; thus the Hermite functions
will be used as the approximation basis. To illustrate the idea, we set $k(X)=\beta X$. We approximate $u$, i.e., we recover the solution
\begin{align*}
  u(t,X) \sim \sum_{n=0}^{N-1} c_n(t) \varphi_{n}(X),
\end{align*}
by attempting to find a sparse representation of the coefficients $c_n$. Thus we are in fact approximating the function $\tilde{f}=e^{\frac{-y^2}{2}}u^2(t,y)$. We are interested in the second moment of the solution, i.e.
\[Q= \E u^2(t,X) \propto \int e^{\frac{-x^2}{2}}u^2(t,x)dx.\]
For each collocation points, one has to solve the ODE to get the information $u(t,y_i)$. The numerical
convergence results for the quantity of interest $Q$ is shown in Figure \ref{fig:ode_recovery} with $\beta=-0.65$ and $t=1$. In this case, the Gaussian quadrature subsampling strategy works very well.

\begin{figure}[h]
\begin{center}
\includegraphics[width=0.49\textwidth]{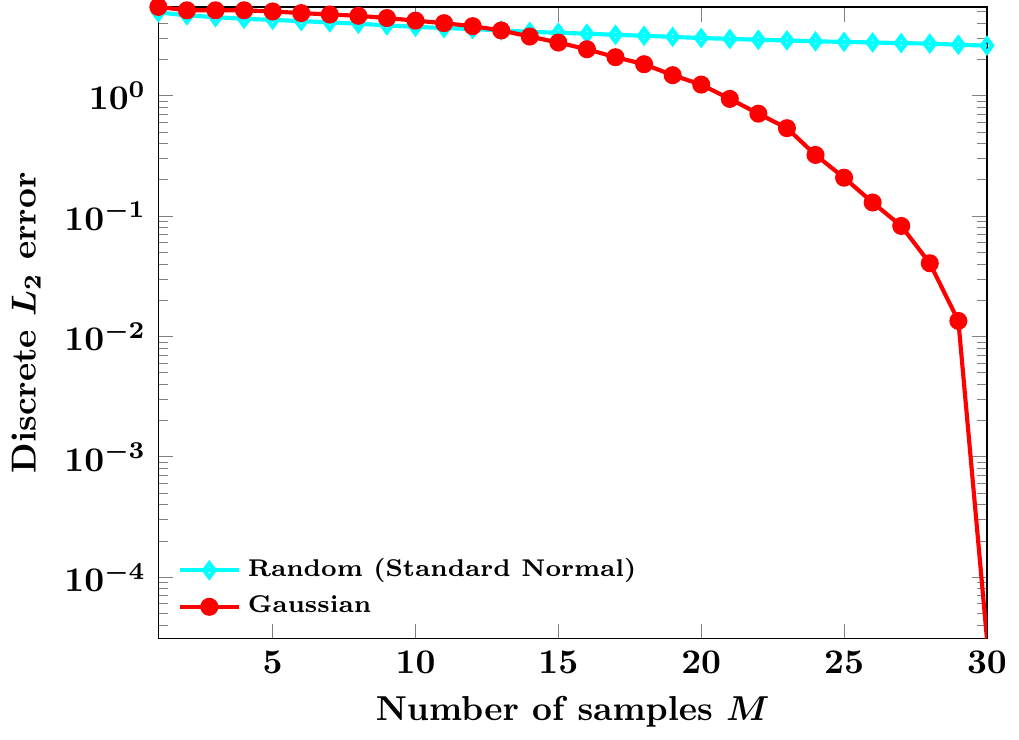}
\end{center}
\caption{Convergence rate against number of sample points. $\beta=-0.65$, with $n=30$, with error averaged over 500 trials.}
\label{fig:ode_recovery}
\end{figure}

\subsection{PDE with random inputs}
This section illustrates the computational performance of our algorithm for the following stochastic linear elliptic problems in two spatial dimensions.
\begin{equation}
\left\{
\begin{array}
{ll}%
-\nabla\cdot(a(\bs{y}, \omega)\nabla u(\bs{y}, \omega))=f(\bs{y},\omega) &\text{in} \ \ \ \mathscr{D}\times\Omega,\\
u(\bs{y},\omega)=0 & \text{on} \ \ \ \partial \mathscr{D} \times \Omega,\\
\end{array}
\right.  \label{spde}%
\end{equation}
with spatial domain $\mathscr{D}=[0,1]^2$. For these numerical examples we take a deterministic load
$f(\boldsymbol{y},\omega)=\cos(y_1)\sin(y_2)$ and construct the random diffusion coefficient $a_N(\bs{y},\omega)$ with one-dimensional
spatial dependence as in \cite[p. 2336]{BNT}:
\begin{equation*}
log(a_N(\omega,\boldsymbol{y})-0.5)=1+X_1(\omega)\bigg(\frac{\sqrt{\pi}L}{2}\bigg)^{1/2}+\sum\limits_{i=2}^{3}
\zeta_if_i(y)X_i(\omega), \label{nu1}
\end{equation*}
where
\begin{equation*}
\zeta_i:=(\sqrt{\pi}L)^{1/2}\text{exp}\bigg(\frac{-(\lfloor\frac{i}{2}\rfloor\pi
L)^2}{8}\bigg ),\ \ \ \text{if} \ i>1,
 \label{nu2}
\end{equation*}
and
\begin{equation*}
  f_i(\bs{y}):=\left\{
\begin{array}[c]{ll}
\sin\bigg(\frac{\lfloor\frac{i}{2}\rfloor\pi y_1}{L_p}\bigg), &
\text{if} \ i\ \text{even},\\
\cos\bigg(\frac{\lfloor\frac{i}{2}\rfloor\pi y_1}{L_p}\bigg), &
\text{if} \ i\ \text{odd}. \end{array}\right.
 \label{nu3}
\end{equation*}
Here $\{X_i\}_{i=1}^{3}$ are mutually independent and
are each uniformly distributed on the interval $[-1,1]$. Thus a family of
Legendre polynomials is used to approximation as a function of $\bs{X}$. For
$y_1\in [0,1]$, let $L_c=1/64$ be a desired physical correlation length
for $a(\bs{y},\omega)$. Then the parameter $L_p$ and $L$ are
$L_p=\max\{1,2L_c\}$  and $L=\frac{L_c}{L_p}$, respectively.
The deterministic elliptic equations are solved by a standard finite element method
and the spatial domain $\mathscr{D}$ is partitioned into 648 triangles with 1369 unknowns.

As the exact solution is not available, we use a high level sparse grid collocation
method to obtain the reference solution. In Fig. \ref{fig:pde_recovery}, we see plots of error in
$\ell_2$ norm of the mean and standard deviation between the reference and $\ell_1$-minimization for
various sampling techniques as a function of the number of samples $M$.

\begin{figure}[h]
\begin{center}
\includegraphics[width=0.49\textwidth]{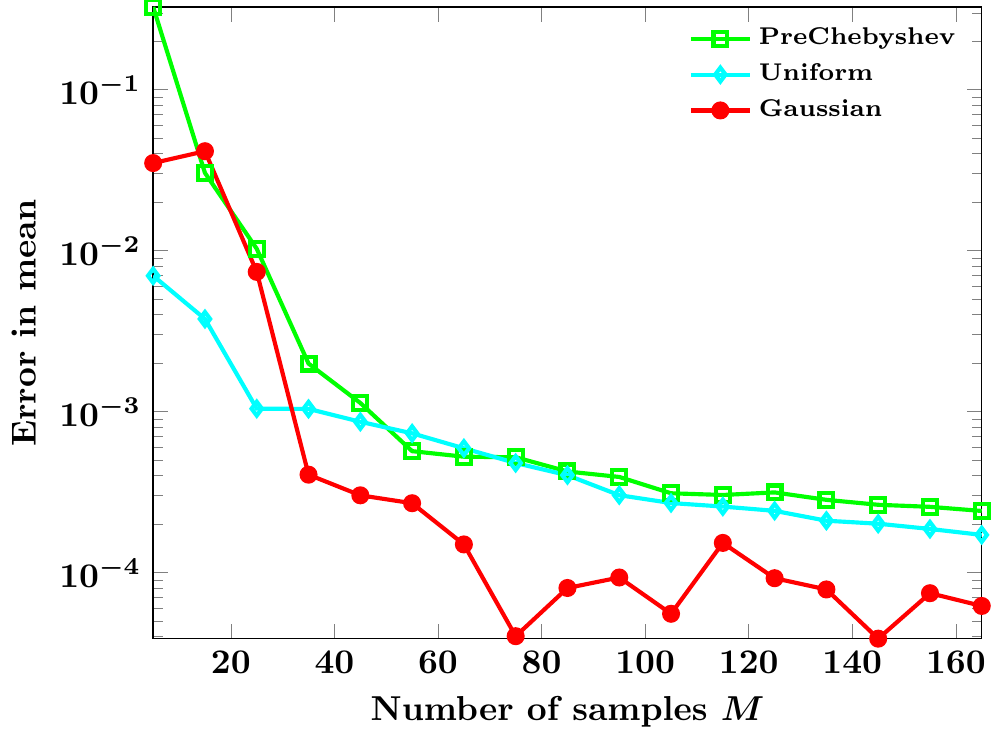}
\includegraphics[width=0.49\textwidth]{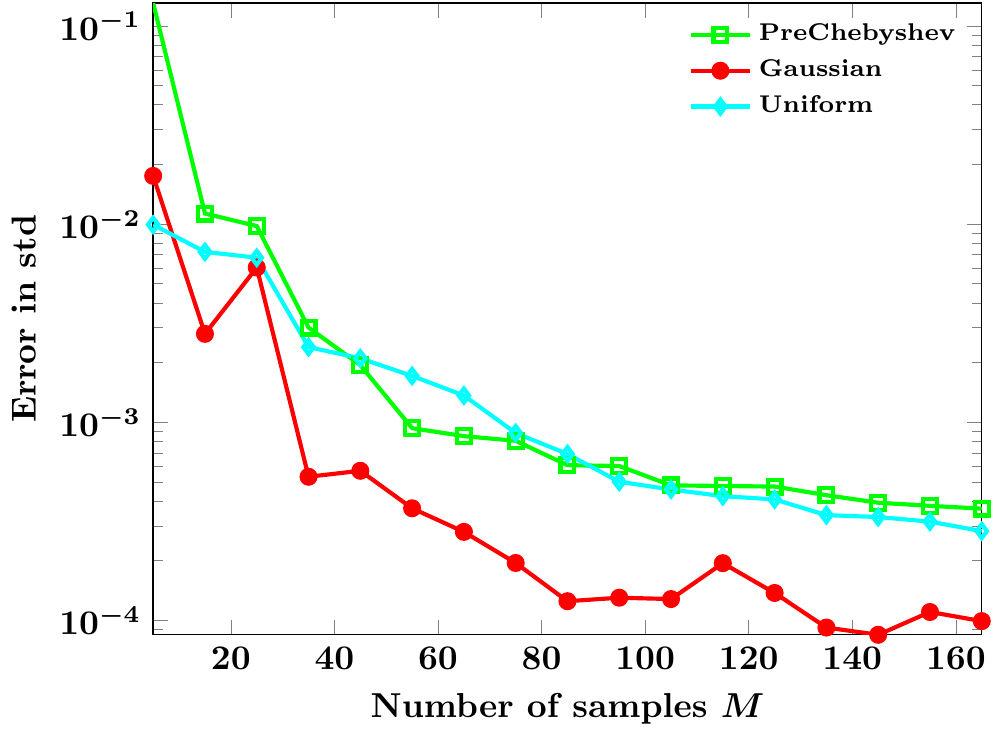}
\end{center}
\caption{Error in $\ell_2$ norm of the mean and standard deviation between the reference and $\ell_1$-minimization for the
various sampling method as a function of the number of samples $M$. $d=3$, $\bs{n}=\bs{11}$ $(N=286)$}\label{fig:pde_recovery}
\end{figure}

\section{Conclusion}
We investigate the problem of approximating a multivariate function via $\ell_1$ minimization methods. Such strategies presume that the underlying function is sparse, or approximately sparse, in the approximating basis. The ability to exactly recover a sparse representation with a small number of samples has been the focus of much research in compressive sampling. We propose sampling by randomly subsampling a tensor-product grid of Gaussian quadrature points. This procedure was investigated in \cite{TangIa} for Legendre polynomials (corresponding to a uniformly distributed random variable). We have provided a nontrivial extension in both analysis and numerical results, covering general Beta distributions taking values on compact domains, as well as one- and two-sided exponential random variable taking values on unbounded domains. In particular, our analysis covers the case of Hermite polynomials (normally-distributed random variables) and Laguerre polynomials (exponential random variables). Our framework provides a non-intrusive way to construct generalized Polynomial Chaos expansions for very general classes of distributions.

\appendix

\section{Proofs}

In this appendix we collect results which imply the results given by the three parts of Lemma \ref{lemma:bounds}. These are essentially well-known results in the theory of orthogonal polynomials. Our analysis uses these well-known results in fairly straightforward ways.

We treat individually the proofs of each of the three Lemmas \ref{lemma:jacobi-bound}, \ref{lemma:full-exponential-bound}, and \ref{lemma:half-exponential-bound}. The strategy for each proof is identical: First we characterize the interval in which the univariate orthogonal polynomial zeros $z^i_j$ are located. Then we use established bounds on weighted orthogonal polynomial families on those intervals.

Note that our main reference for bounds on orthogonal polynomials is \cite{jakeman_2015}, which explicitly states these bounds in the form that we require. However, the conclusions in \cite{jakeman_2015} are essentially a reshuffling of far deeper, more technical, and elegant results in \cite{levin_christoffel_1992,nevai_generalized_1994,levin_orthogonal_2005,levin_orthogonal_2006}.

After proving the three Lemmas \ref{lemma:bounds}, we end with the proof of our main result, Theorem \ref{thm:main-sample-count}.

\subsection{Proof of Lemma \ref{lemma:jacobi-bound}}
Assume the setup of Lemma \ref{lemma:jacobi-bound}: that $X^i$ is a scalar random variable taking values on $[-1,1]$ with a Beta distribution with shape parameters $\gamma, \delta \geq \frac{1}{2}$. We first establish intervals in which the zeros of orthogonal polynomial lie, and then use boundedness results on these intervals. This bounded case is easiest, for which the $n$ zeros $z^i_j$ for $j = 1, \ldots, n$ all lie inside the $n$-independent compact interval of orthogonality.
\begin{lemma}\label{lemma:jacobi-zeros}
  For all $n \in \N$, we have $\left\{ z^i_j\right\}_{j=1}^n \subset [-1,1]$.
\end{lemma}
See, e.g., Theorem 3.3.1 of \cite{szego_orthogonal_1975}. We can now use results from \cite{jakeman_2015}.
\begin{lemma}[\cite{jakeman_2015}]
  For all $n \in \N$ and $0 \leq k \leq n-1$, we have
  \begin{align*}
    \sup_{z^i \in [-1,1]} \left|\psi_{k,n}(z^i)\right|^2 \leq C(\gamma^i,\delta^i)
  \end{align*}
\end{lemma}
Lemma \ref{lemma:jacobi-bound} now follows easily:
\begin{align}\label{eq:jacobi-bound-proof}
  L_i(n) = \max_{0 \leq k \leq n-1} \max_{j=1,\ldots, n} \left| \psi_{k,n}(z^i_j) \right|^2 \leq
  \max_{0 \leq k \leq n-1} \sup_{z^i \in [-1,1]} \left| \psi_{k,n}(z^i) \right|^2 \leq C(\gamma^i, \delta^i).
\end{align}

\subsection{Proof of Lemma \ref{lemma:full-exponential-bound}}
Both of the exponential cases are more subtle. Assume the setup of Lemma \ref{lemma:full-exponential-bound}, that the random variable $X^i$ has a two-sided exponential distribution, $\rho^i\left(x^i\right) = \exp (-|x^i|^\alpha)$ with $\alpha > \frac{3}{2}$. In order to characterize intervals containing the zeros of the associated orthogonal polynomials, we will need the numbers $a^W_n$, which for $n \geq 1$ are given by
\begin{align*}
  a^W_n = \left(n \frac{ \sqrt{\pi} \Gamma\left(\frac{\alpha}{2}\right) }{\Gamma \left(\frac{ \alpha}{2} + \frac{1}{2}\right) } \right)^{1/\alpha}.
\end{align*}
Note that, modulo an $\alpha$-dependent constant, these numbers scale like $n^{1/\alpha}$. The $a_n$ are the \textit{Mhaskar-Rahkmanov-Saff} numbers associated to the weight $\sqrt{\rho^i}$ \cite{mhaskar_where_1985} and play an essential role in the anaylsis of weighted polynomials.
\begin{lemma}[\cite{levin_christoffel_1992}]
  For each $n \in \N$, the $n$-point Gaussian quadrature nodes satisfy:
  \begin{align*}
    \left\{ z^i_j \right\}_{j=1}^{n} \subset \left[ -\widehat{a}^W_n, \widehat{a}^W_n \right],
  \end{align*}
  where $\widehat{a}^W_n$ satisfy
  \begin{align}\label{eq:C-def}
    \widehat{a}^W_n = a_n \left[ 1 + c n^{-2/3} \right],
  \end{align}
  with $c$ a $n$-independent constant.
\end{lemma}
If $\alpha$ is an even integer, we can take $C = 0$ so that $\widehat{a}^W_n = a^W_n$ \cite{mate_asymptotics_1986}. We have bounds for weighted polynomials on the interval defined above.
\begin{lemma}[\cite{jakeman_2015}]
  Let $c > 0$ be as in \eqref{eq:C-def} defining $\widehat{a}^W_n$. For all $n \in \N$ and $0 \leq k \leq n-1$, we have
  \begin{align*}
    \sup_{z^i \in [-\widehat{a}^W_n, \widehat{a}^W_n]} \left| \psi_{k,n}(z^i)\right|^2 \leq C(\alpha) n^{2/3}.
  \end{align*}
\end{lemma}
With these two lemmas, essentially the same argument as in \eqref{eq:jacobi-bound-proof} can be used to establish,
\begin{align*}
  L_i(n) = \max_{0 \leq k \leq n-1} \max_{j=1,\ldots, n} \left| \psi_{k,n}(z^i_j) \right|^2 \leq C n^{2/3},
\end{align*}
which is Lemma \ref{lemma:full-exponential-bound}.

\subsection{Proof of Lemma \ref{lemma:half-exponential-bound}}
  Assume the setup of Lemma \ref{lemma:half-exponential-bound}, that the random variable $X^i$ has a one-sided exponential distribution with density $\rho^i\left(x^i\right) \propto \exp(-|x^i|^\alpha)$ with $\alpha > \frac{3}{4}$.

    In this half-line case, both the strategy and the results are essentially the same as with the whole real line in the previous section. One major change is in the constants $a_K$, which in this case are given by
\begin{align*}
  a^H_n = \left(n \frac{ \sqrt{\pi} \Gamma\left(\alpha\right) }{\Gamma \left(\alpha + \frac{1}{2}\right) } \right)^{1/\alpha}.
\end{align*}
Note again that these numbers scale like $n^{1/\alpha}$. These $a^H_n$ are the \textit{Mhaskar-Rahkmanov-Saff} numbers associated to the half-line weight $\sqrt{\rho^i}$.
\begin{lemma}[\cite{levin_christoffel_1992}]
  For each $n \in \N$ with $n \geq 1$, the $n$-point Gaussian quadrature nodes satisfy:
  \begin{align*}
    \left\{ z^i_j \right\}_{j=1}^{n} \subset \left[ 0, \widehat{a}^H_n \right],
  \end{align*}
  where $\widehat{a}^H_n$ satisfy
  \begin{align}\label{eq:C-def-half}
    \widehat{a}^H_n = a_n \left[ 1 + c n^{-2/3} \right],
  \end{align}
  with $c$ a $n$-independent constant.
\end{lemma}
We have bounds for weighted polynomials on the interval defined above.
\begin{lemma}[\cite{jakeman_2015}]
  Let $C > 0$ be as in \eqref{eq:C-def-half} defining $\widehat{a}^H_n$. For all $n \in \N$ and $0 \leq k \leq n-1$, we have
  \begin{align*}
    \sup_{z^i \in [0, \widehat{a}^H_n]} \left| \psi_{k,n}(z^i)\right|^2 \leq C(\alpha) n^{2/3}.
  \end{align*}
\end{lemma}
Again the same argument as in \eqref{eq:jacobi-bound-proof} can be used to establish,
\begin{align*}
  L_i(n) = \max_{0 \leq k \leq n-1} \max_{j=1,\ldots, n} \left| \psi_{k,n}(z^i_j) \right|^2 \leq C n^{2/3},
\end{align*}
which is Lemma \ref{lemma:half-exponential-bound}.

\subsection{Proof of Theorem \ref{thm:main-sample-count}}\label{sec:main-theorem}

Under the assumptions of this theorem, we perform the algorithm given in Section \ref{sec:gpc-stuff:algorithm}. Recalling the notation, we have a multi-index set $\Lambda \subset \Lambda^P_{\bs{n}-\bs{1}}$ for some multi-index $\bs{n}$. The index set $\Lambda$ has size $N$. We construct the tensor-product Gauss quadrature rule with $n_i$ points in dimension $i$. Then the rectangular $\left(\prod_{i=1}^d n_i\right) \times N$ weighted Vandermonde-like matrix $\bs{A}$ is defined, whose entries are
\begin{align}\label{eq:A-def}
  (A)_{p,q} &= \sqrt{w_{\bs{\ell}(p)}} \phi_{\bs{k}(q)}\left( \bs{z}_{\bs{\ell}(p)}\right), & 1 &\leq p, q \leq \prod_{i=1}^d n_i
\end{align}
where $\bs{k}(q)$ and $\bs{\ell}(p)$ represent any enumeration of the elements in $\Lambda^P_{\bs{n} - \bs{1}}$. The matrix $\bs{A}$ is an orthogonal matrix. (See Lemmas \ref{lemma:general-psi-orthonormality} or \ref{lemma:D-orthogonality}, and Definition \ref{def:rectangular-orthogonal-matrix}.) According to Section \ref{sec:gpc-stuff:algorithm}, $\bs{D}$ is formed by subsampling rows from $\bs{A}$.

We subsample rows from $\bs{A}$ without replacement. This does \textit{not} sample iid from $\nu_{\bs{n}}$ since the samples are dependent. Thus, Theorem \ref{RIP_bos} cannot be used directly to analyze this procedure. Nevertheless, the analysis may be amended to include this type of procedure; see Corollary 12.38 in Section 12.6 of \cite{rauhut_mathematical_2013} and the surroundng discussion.

\begin{lemma}[\cite{rauhut_mathematical_2013}]\label{lemma:row-subsample}
  Suppose $\bs{A}$ is an orthogonal matrix whose entries are bounded by $\frac{1}{\sqrt{\prod_{i=1}^d n_i}} \sqrt{L}$. Suppose we subsample $M$ rows from $\bs{A}$. Then under the sample count condition \eqref{num_RIP}, the subsampled matrix satisfies the conclusions of Theorem \ref{RIP_bos}.
\end{lemma}

We then must determine how to bound the entries of the matrix $\bs{A}$. Assuming that the marginal components of the random variable $\bs{X}$ each have a distribution satisfying any of the conditions in Lemmas \ref{lemma:bounds}, then the entries of $\bs{A}$, defined by \eqref{eq:A-def}, satisfy
\begin{align*}
\left(\prod_{i=1}^d \sqrt{n_i}\right) \left|\left(A\right)_{p,q}\right| &= \left(\prod_{i=1}^d \sqrt{n_i}\right) \sqrt{w_{\bs{\ell}(p)}} \left|\phi_{\bs{k}(q)}\left( \bs{z}_{\bs{\ell}(p)} \right)\right| \\
                                                             &\stackrel{\eqref{gpcbais},\eqref{eq:tp-weights}}{=} \prod_{i=1}^d \sqrt{n_i w^i_{\ell(p)_i}} \left|\varphi^i_{k(q)_i}\left( z^i_{\ell(p)}\right)\right| \\
                                                             &\stackrel{\eqref{eq:quadratureweight}}{=} \prod_{i=1}^d \sqrt{n_i \lambda^i_{n_i}\left(z^i_{\ell(p)}\right)} \left|\varphi^i_{k(q)_i}\left( z^i_{\ell(p)}\right)\right| \\
                                                             &\stackrel{\eqref{eq:psi-univariate-def}}{=} \prod_{i=1}^d \left|\psi_{\ell(p)_i,n_i}\right| \\
                                                             &\stackrel{\textrm{Lemmas }\ref{lemma:bounds}}{\leq} \prod_{i=1}^d \sqrt{L_i(n_i)},
\end{align*}
with the individual $L_i$ factors given by the bounds in Lemmas \ref{lemma:bounds} depending on the distribution of $\bs{X}$. Thus,
\begin{align*}
  \sup_{p,q} \left(\prod_{i=1}^d n_i\right) \left(A\right)^2_{p,q} \leq L(\bs{n})
\end{align*}
Thus, under condition \eqref{eq:sample-count-criterion}, then Lemma \ref{lemma:row-subsample} implies the conclusion of Theorem \ref{thm:main-sample-count}.

\bibliographystyle{abbrv}
\bibliography{cs_gaussian1212}

\end{document}